\documentclass[a4paper,11pt]{article}

\usepackage[utf8]{inputenc}
\usepackage{amsfonts,amsmath,amsthm,amssymb,epsfig,stmaryrd,psfrag, mathtools}

\newtheorem{theorem}{Theorem}
\newtheorem{proposition}[theorem]{Proposition}

\newtheorem{lemma}[theorem]{Lemma}

\def\N{{\mathbb N}}
\def\EE{{\mathbb E}}
\def\PP{{\mathbb P}}

\def\ind{{\mathbf{1}}}

\def\d{\partial}

\def\bG{\mathbb{G}}

\def\Ecal{{\mathcal E}}

\def\Mcal{{\mathcal M}}

\def\Fcal{{\mathcal F}}

\def\EE{{\mathbb E}}

\def\PP{{\mathbb P}}

\def\ind{{\mathbf{1}}}

\def\R{{\mathbb R}}

\def\N{{\mathbb N}}

\newcommand{\pp}{{p\in\mathbf{P}}}
\newcommand{\on}{\operatorname}
\def\mfrak{\mathfrak{m}}

\newcommand{\BEAS}{\begin{eqnarray*}}
\newcommand{\EEAS}{\end{eqnarray*}}
\newcommand{\BEA}{\begin{eqnarray}}
\newcommand{\EEA}{\end{eqnarray}}
\newcommand{\BEQ}{\begin{equation}}
\newcommand{\EEQ}{\end{equatioyn}}
\newcommand{\BIT}{\begin{itemize}}
\newcommand{\EIT}{\end{itemize}}
\newcommand{\BNUM}{\begin{enumerate}}
\newcommand{\ENUM}{\end{enumerate}}

\title{Replica Bounds by Combinatorial Interpolation\\ for Diluted Spin Systems}

\author{Marc Lelarge \\
\emph{INRIA-ENS} \\
\emph{Paris, France}\\
marc.lelarge@ens.fr
\and 
Mendes Oulamara\\
\emph{École Normale Supérieure} \\ 
\emph{Paris, France}\\
mendes.oulamara@ens.fr
}
\date{}

\begin{document}
\maketitle

\begin{abstract}
In two papers Franz, Leone and Toninelli proved bounds
for the free energy of diluted random constraints satisfaction problems, for
a Poisson degree distribution \cite{physicpo} and a general distribution \cite{physic}. 
Panchenko and Talagrand \cite{diluted} simplified the proof and generalized the result of \cite{physicpo} for the Poisson case.
We provide a new proof for the general degree distribution case and as a corollary, we obtain new bounds for the size of the largest independent set (also known as hard core model) in a large random regular graph.
Our proof uses a combinatorial interpolation based on biased random walks \cite{justin} and allows to bypass the arguments in \cite{physic} based on the study of the Sherrington-Kirkpatrick (SK) model.

\textbf{keywords}: interpolation method, Parisi formula, configuration model, random walks, free energy, hard-core model
\end{abstract}

\section{Introduction}

We consider diluted spin glass models where particles interact through a Hamiltonian defined on a sparse random
graph, i.e. the number of interactions remains of order one when the size of the system tends to infinity.
Once we fix the probability law generating the random graph and the 
Hamiltonian, natural questions arise: Does the normalized free energy have a
limit as the size of the graph, $N$, tends to infinity? In such a case, is it possible to get an analytic formula for this limit?
In this paper, we provide a general upper bound on the possible limit. Our proof relies on a new variation of the interpolation method adapted to our setting.

The interpolation method has been introduced by Guerra and Toninelli
\cite{guerra2002thermodynamic,guerra2003broken,guerra2012fluctuations} to study the convergence and bounds on the limit of the free energy 
of the SK model \cite{skmodel} and other mean-field spin glass models
on complete graphs. 
These ideas originally used to study fully connected random graph
have been extended by various authors to the study of diluted (or sparse) random graph. For instance, Bayati, Gamarnik and Tetali \cite{bayati2010combinatorial}
showed the existence of a limit for various models of Hamiltonian on the Erdös-Renyi and the $d$-regular random
graphs. Abbe and Montanari used it \cite{abbe2013conditional} to show the convergence of conditional 
entropy in the context of coding theory.
In \cite{justin} Salez devised a discrete interpolation method based on random walks (hence there is no need of a continuous
parameter anymore) well suited to combinatorial models of random graphs with prescribed degree in order to show the existence of the limit for a wide range of models.

In \cite{physicpo}, Franz and Leone proved an asymptotic bound, linked to the Parisi formula \cite{parisi1,parisi2},
on the free energy for the $p$-spin and the $k$-SAT models on random graphs with a Poisson degree distribution.
To do so, they interpolate between a random hypergraph, and a graph in which every hyperedge of size $k$ is replaced by
$k$ independent \emph{sites}, by varying continuously the rate of the Poisson distribution. Panchenko and Talagrand 
noticed \cite{diluted} that the proof can be generalized to a wider class of Hamiltonians verifying some conditions.

Franz, Leone and Toninelli also published a proof \cite{physic} of the same bound for random graphs with a general degree
distribution. That proof uses advanced results coming from the study of the SK model \cite{skmodel},
such as Ghirlanda-Guerra identities and Hamiltonian gaussian perturbation. They use a discrete
deterministic interpolation, where at each time-step, they delete one edge and add $k$ sites.

In the setting of coding theory, Montanari \cite{montanari2004} used approximations based on Poisson distributions to approximate
a general degree distribution.

In our paper, we use the formalism of \cite{diluted}, with similar weak hypotheses on the Hamiltonian model, and the idea
of a random discrete interpolation as in \cite{justin} to prove the Parisi asymptotic bound on the free energy for a very 
general class of graphs, where the degree distribution as well as the distribution of the size of the hyperedges are prescribed.
Our contribution can be seen as doing what Panchenko and Talagrand \cite{diluted} did for the Poisson distribution case \cite{physicpo}, 
but for the general degree distribution case \cite{physic}.
The structure of the interpolation produces a very natural proof for our combinatorial graph model which encompass most of the cited models. It only uses basics of the theory of martingales.

To illustrate our bound, we provide explicit calculations for the hard-core model on regular graphs. 
The ground-breaking work \cite{ding2016maximum} shows that for sufficiently large degrees, the bound given by the one step replica symmetry breaking (1-RSB) is the exact value of the asymptotic size of a maximum independent set in a random regular graph. As a corollary of our result, we prove that this 1-RSB formula is an upper bound for the size of the maximum independent set for all degrees. 
For small degrees, our bounds have been numerically computed in \cite{hcrevisited} and improve on the best known rigorous upper bounds given in  \cite{mckay1987independent,hoppen2013local}.
Note that for small degrees, it is expected that the exact limit for the maximum size of the independent set will not anymore be the 1-RSB formula but the full RSB formula \cite{hcrevisited} (both formula are the same when the degree is sufficiently large). Also, we did not compute numerically this full RSB formula, our main general result shows that the full RSB formula is an upper bound on the size of the maximum independent set. Showing that this bound is tight for all degrees is a challenging open problem.

\section{Model}
Let $\mathbf{P}$ be a set of integers greater or equal to 2.
We consider
multigraphs of the form $G=(\mathbf{V} , (\mathbf{E}_p)_{p\in\mathbf{P}})$ where 
$\mathbf{V} = \llbracket 1, N \rrbracket$ for some $N$, 
is a set of vertices, and for any $p\in\mathbf{P}$, $\mathbf{E}_p$  
is a set of $p$-edges, i.e. each $e\in \mathbf{E}_p$ contains $p$ vertices (not necessarily distinct) $e(1),e(2),\dots, e(p)\in\mathbf{V}$ and we denote $\d e=\{e(1),\dots, e(p)\}$ the (multi-)set of these vertices. The $\ell$-th element of $\mathbf{E}_p$ will be denoted by $e^p_\ell$. Note that an edge can appear with some multiplicity: $\d e^p_{\ell_1} = \d e^p_{\ell_2}$ for $\ell_1 \neq \ell_2$.

On such graphs, the space of spin configurations is denoted by $\Sigma_N=\{-1,1\}^\mathbf{V}$. We consider independent random functions $(\theta_p)_{p\in\mathbf{P}}$ where $\theta_p:\{-1,1\}^{p}\mapsto \mathbb{R}$ and for each $\pp$, a sequence $(\theta_{p,e})_{e\in {\bf E}_p}$ of i.i.d. copies of $\theta_p$ where the sequences $(\theta_{p,e})_{e\in {\bf E}_p}$ and $(\theta_{p',e})_{e\in {\bf E}_{p'}}$ are independent for $p\neq p'$. Let $h:\{-1,1\}\to \R$ be a random function and $(h_i)_{i\in \mathbf{V}}$ be i.i.d. copies of the function $h$
(note that we have $h(\sigma)=\mu \sigma+\nu$ for some random $\mu,\nu$).

We define the following Hamiltonian on the graph G
for $\sigma \in \Sigma_N$~:
\begin{equation}
-H_G(\sigma)= \sum_{\pp} \left( \sum_{e\in \mathbf{E}_p} \theta_{p,e}(\sigma_{\d e}) \right)
                              + \sum_{i\in \mathbf{V}}h_i(\sigma_i),
\end{equation}
where $\sigma_{\d e}=(\sigma_i)_{i\in \d e}$.
\noindent
As in \cite{diluted}, we make the following assumptions on the random functions $\theta_p$. For each $\pp$, we assume that there is a random function $f_p:\{-1,1\}\mapsto\mathbb{R}$ with
i.i.d. copies $f_{p,1}, \ldots, f_{p,p}$ and two random variables $a_p,b_p$ independent
of the previous functions, satisfying the conditions $\forall \sigma_1,\dots,\sigma_p \in\{-1,1\}^p$:
\begin{equation}
\label{cond1}
\begin{split}
\exp\theta_p(\sigma_1,\ldots,\sigma_{p}) &= a_p(1+b_pf_{p,1}(\sigma_1)\ldots f_{p,p}(\sigma_{p})), \\
\forall n\ge 1, \mathbb{E}[(-b_p)^n ] & \ge 0, \\
|b_pf_{p,1}(\sigma_1)\ldots f_{p,p}(\sigma_{p})| & < 1 \text{ a.s.}
\end{split}
\end{equation}
In addition, we assume that there is a constant $\kappa>0$ such that:
\begin{equation}
\label{cond2}
\forall\pp,\: |\theta_p|  \le\kappa \mbox{ and }|h|  \le\kappa \mbox{, a.s.}
\end{equation}
Finally, we also assume that for any 
$\pp$
at least one of the following conditions is satisfied~:
\begin{equation}
\label{cond3}
p \text{ is even \quad or \quad } f_p\ge0 \text{ a.s.}
\end{equation}
We now define the sequence of random graphs that we will consider. For fixed degrees $d=(d_i)_{i\in \llbracket 1, N \rrbracket}\in \N^N$ and  edge cardinalities $E=(E_p)_\pp\in \N^{\mathbf{P}}$, we define $\bG(d, E)$ the random graph built according to the configuration model: it is drawn uniformly among all multi-graphs with $N$ vertices, exactly $E_p$ $p$-edges for each $\pp$ and such that the $i$-th vertex has degree $d_i$. Note that for this set to be non-empty, we need to have $\sum_{i=1}^N d_i = \sum_{\pp}pE_p$.

We will consider two sequences $d^N=(d^N_i)_{1\le i\le N}$ and $E^N=(E^N_p)_\pp$ for $N\in \mathbb{N}$
and assume that the empirical distributions of these sequences tend to two probability measures $\mu$ on $\N$ and $\nu$ on $\bf P$ in the following strong sense:
\BEA
\label{cond:deg1}\forall N\in \N, \quad \sum_{i=1}^N d^N_i &=& \sum_{\pp}pE^N_p\\
\label{cond:deg2}\forall k\in\N, \mu_N(k) \coloneqq \frac{1}{N}\sum_{i=1}^N\ind(d^N_i=k) && 
    \lim_{N\to\infty}\sum_{k\in\mathbb{N}}k|\mu_N(k)-\mu(k)| = 0\\
\label{cond:deg3}\forall\pp, \nu_N(p) \coloneqq \frac{E^N_p}{\sum_{q\in\mathbf{P}}E^N_q}&&
    \lim_{N\to\infty}\sum_{\pp}p|\nu_N(p)-\nu(p)| = 0
\EEA
In addition, we assume that
\BEA
\label{cond:deg4}\sup_{N\geq 1} \frac{1}{N} \sum_{i=1}^N(d^N_i)^2 <\infty \text{ , }
\sum_{n\in\N}n^2\mu(n) < \infty \text{ and }
\sum_{n\in\N}n\nu(n) < \infty
\EEA
Note that under this last assumption, the probability for our graph to be simple stays bounded away from zero as $N$ tends to infinity \cite{svante1,svante2}.

For such sequences, we define $G^N=\bG(d^N, E^N)$  a sequence of random graphs and the associated free energy:
\BEA
\label{eq:defFN}F_N = \frac{1}{N}\mathbb{E}\log\sum_{\sigma\in \Sigma_N}\exp\left( -H_{G^N}(\sigma)\right),
\EEA
where $\mathbb{E}$ is the expectation with respect to the randomness of the graph and Hamiltonian.
We also define the probability distribution $\rho$ on $\mathbf{P}$ corresponding to the size biased distribution of $\nu$:
\BEAS
\forall \pp, \quad \rho(p) = \frac{p\nu(p)}{\sum_q q\nu(q)}.
\EEAS

{\bf Application to the hard-core model on $d$-regular graphs:}\\
To illustrate our results, we will consider the hard-core (or independent set) model on a $d$-regular graph $G=(\mathbf{V},\mathbf{E})$ with fugacity $\lambda > 1$. An independent set $I\subset V$ in a graph $G$ is a subset of the vertices such that if $v_1,v_2\in I$ then there is no edge between $v_1$ and $v_2$ in $G$. We give a weight $\lambda^{|I|}$ to such set where $|I|$ is the size of $I$. 

This model \emph{per se} does not verify Conditions (\ref{cond1},\ref{cond2}), therefore we relax it by a parameter $A>0$ and we will show later that we can
make $A$ tend to $+\infty$ to get the actual hard-core model. $A$ corresponds to an energy cost for each edge violating the constraint by connecting two vertices of the independent set.
We define the following Hamiltonian:
  \BEAS
\exp\left(-H_G(\sigma)\right) = \lambda^{\sum_{i\in \mathbf{V}} \frac{1+\sigma_i}{2}}\prod_{(i,j)\in \mathbf{E}}\left(1-(1-e^{-A})\frac{(1+\sigma_i)(1+\sigma_j)}{4}\right).
\EEAS
A site $i$ with $\sigma_i=1$ (resp. $\sigma_i=-1$) is called occupied (resp. unoccupied).
This case corresponds to $\mu(\{d\})=1$ since the graph is $d$-regular, $\mathbf{P}=\{2\}$ since there are only 2-edges, and
\BEAS
a_2=1,\quad b_2 = -\frac{1-e^{-A}}{4},\quad f_2(\sigma) = 1+\sigma, \quad h(\sigma)= \frac{\log \lambda}{2}(1+\sigma).
\EEAS
In particular, conditions (\ref{cond1},\ref{cond2},\ref{cond3}) are satisfied with $\kappa =\max(\log \lambda , A)$.

Let $\mathcal{I}(G_N)$ be the set of all independent sets of $G_N$ a random $d$-regular graph and $\EE$ is the expectation with respect to the randomness of the graph.
Then it is easy to check that 
$$
\frac{\log \lambda}{N}\EE\left[\max_{I\in \mathcal{I}(G_N)}|I|\right]
 \leq
\frac{1}{N}\EE\log \sum_{I\in \mathcal{I}(G_N)}\lambda^{|I|} \leq F_N. 
$$
Since \cite{bayati2010combinatorial} shows the existence of the following limit:
\BEAS
\lim_{N\to \infty} \frac{1}{N}\EE\left[\max_{I\in \mathcal{I}(G_N)}|I|\right]=\alpha^*,
\EEAS
we have
\begin{equation}
\label{eq:approxhc}
\alpha^* \log \lambda \le \liminf_{N\to \infty} F_N.
\end{equation}
As a result, we see that an upper bound on $F_N$ directly translate into an upper bound on $\alpha^*$.

\section{Main Results}

In order to state our result, we need to introduce another notation taken from \cite{diluted}. Given a function $f:\{-1,+1\}^p\to \R$ and a vector of real numbers $x=(x_1,\ldots, x_p)$, we define for $\sigma\in \{\pm 1\}$,
\BEAS
\langle f\rangle_x^-(\sigma) = \frac{\sum_{\epsilon_1,\dots, \epsilon_{p-1}=\pm 1} f(\epsilon_1,\dots,\epsilon_{p-1},\sigma) \exp\sum_{\ell=1}^{p-1} x_\ell \epsilon_\ell}{\sum_{\epsilon_1,\dots, \epsilon_{p-1}=\pm 1} \exp\sum_{\ell=1}^{p-1} x_\ell \epsilon_\ell},
\EEAS
and
\BEAS
\langle f\rangle_x = \frac{\sum_{\epsilon_1,\dots, \epsilon_{p}=\pm 1} f(\epsilon_1,\dots,\epsilon_{p-1},\epsilon_p) \exp\sum_{\ell=1}^{p} x_\ell \epsilon_\ell}{\sum_{\epsilon_1,\dots, \epsilon_{p}=\pm 1} \exp\sum_{\ell=1}^{p} x_\ell \epsilon_\ell}.
\EEAS
Let us define for each $\pp$, the random function:
\BEAS
\mathcal{E}_p(\epsilon_1,\dots, \epsilon_p) = \exp(\theta_p(\epsilon_1,\dots, \epsilon_p)),
\EEAS
so that under Condition (\ref{cond1}), we have
\BEA
\label{eq:defEx}\langle\mathcal{E}_p\rangle_x^-(\sigma) = a_p\left( 
    1+b_p f_{p,p}(\sigma)\prod_{1\le l\le p-1}
        \frac{\on{Av}f_{p,l}(\epsilon)\exp(x_l\epsilon)}{\on{ch}(x_l)}
    \right),
    \EEA
    where $\on{Av}$ means average over $\epsilon =\pm 1$, and,
\BEAS
\langle\mathcal{E}_p\rangle_x = a_p\left( 
    1+b_p\prod_{1\le l\le p}
        \frac{\on{Av}f_{p,l}(\epsilon)\exp(x_l\epsilon)}{\on{ch}(x_l)}
    \right).
\EEAS
Finally, since 
$\left.\langle\mathcal{E}_p\rangle_x^-\right(\sigma)$ is positive as a consequence of (\ref{cond1}),
we define
\BEA
\label{eq:defU}U_p(\theta_p,x_1,\dots,x_{p-1},\sigma) = \log
\left.\langle\mathcal{E}_p\rangle_x^-\right(\sigma).
\EEA
Given an arbitrary distribution $\zeta$ on $\R$, we consider an i.i.d. sequence $x^p_{i,\ell}$ for $\pp, i,\ell\geq 1$ with distribution $\zeta$ 
and $(\theta_{p,i})_{i\geq 1}$ i.i.d. copies of $\theta_p$, and define for $\pp$ and $i\geq 1$,
\BEA
\label{eq:defUi}U_{p,i}(\sigma;\zeta) = U_p(\theta_{p,i}, x^p_{i,1},\dots, x^p_{i,p-1},\sigma).
\EEA

\subsection{Replica Symmetric Bound (RS)}
\begin{theorem}
    \label{thmrs}
    If conditions (\ref{cond1},\ref{cond2},\ref{cond3}) and (\ref{cond:deg1},\ref{cond:deg2},\ref{cond:deg3},\ref{cond:deg4}) are satisfied, then for any distribution $\zeta$ on $\R$, we have
\begin{multline}
\label{eq:thmrs}
F_N \le 
\mathbb{E}\left[\log\left(\sum_{\sigma=\pm 1}\exp\left( \sum_{i=1}^d U_{p_i,i}(\sigma;\zeta) + h(\sigma) \right)\right) \right] \\
    -\mathbb{E}[d]\mathbb{E}\left[\frac{p_1-1}{p_1}\log\langle \mathcal{E}_{p_1}\rangle_x\right]
    + o_N(1)
\end{multline}
where $d$ is a random variable with law $\mu$, $(p_i)_{i\ge 1}$ is a sequence of i.i.d. random variables with law $\rho$ 
    and $x=(x_i)_{i\ge 1}$ is a sequence of i.i.d. real random variables with distribution $\zeta$.
    \end{theorem}

{\bf Application to the hard-core model on $d$-regular graphs:}\\
We have:
  \BEAS
  \langle \mathcal{E}_2\rangle_x^-(\sigma) &=& 1-(1-e^{-A})\frac{1+\sigma}{2(e^{-2x_1}+1)}\\
  \langle \mathcal{E}_2\rangle_x &=& 1-\frac{1-e^{-A}}{(e^{-2x_1}+1)(e^{-2x_2}+1)}
  \EEAS
  Hence the right-hand term in Theorem \ref{thmrs} is given by:
  \begin{multline*}
F_N \leq  
\EE\left[ \log\left( 1+\lambda \prod_{i=1}^d \frac{e^{-2x_i}+e^{-A}}{1+e^{-2x_i}}\right)\right] \\
        -\frac{d}{2} \EE\left[\log\left( 1-\frac{1-e^{-A}}{(1+e^{-2x_1})(1+e^{-2x_2})}\right)\right]+o_N(1),
    \end{multline*}
where $x_1,\dots$ is a sequence of i.i.d. random variables with law $\zeta$. 
We can make the change of variable $\pi = \frac{1}{1+e^{-2x}}$. Moreover, only the error term on the right hand side depends on $N$, hence we can make $N$ tend to $+\infty$ to get:
\begin{multline*}
\limsup_{N \to \infty} F_N \leq  \EE\left[ \log\left( 1+\lambda \prod_{i=1}^d (1-(1-e^{-A})\pi_i)\right)\right] \\ -\frac{d}{2} \EE\left[ \log\left(1-(1-e^{-A})\pi_1\pi_2\right)\right],
\end{multline*}
  where $\pi_1,\dots$ is a sequence of i.i.d. random variables on $(0,1)$. The expression inside the expectations is easily dominated
  and we can push $A$ to $+\infty$. Hence by (\ref{eq:approxhc}):
\BEA
\alpha^* \log\lambda \leq 
 \EE\left[ \log\left( 1+\lambda \prod_{i=1}^d (1-\pi_i)\right)\right] -\frac{d}{2} \EE\left[ \log\left(1-\pi_1\pi_2\right)\right].
\EEA
   
In order to get the tightest bound, we should minimize the bound on $\alpha^*$ with respect to $\lambda$ and the distribution of $\pi$. To get 
an explicit formula, consider the case where the $\pi_i$'s are deterministic: $\pi_i=\pi \in (0,1)$.
We define
$\Phi(\lambda,\pi,\alpha) = \log \left(1+\lambda(1-\pi)^d\right)-\frac{d}{2}\log\left(1-\pi^2\right)-\alpha\log \lambda$, its minimal
value when $\alpha$ is fixed $\Phi_d(\alpha)=\inf_{\lambda,\pi}\Phi(\lambda,\pi,\alpha)$ and
$\alpha_{RS} = \inf\{\alpha>0,\: \Phi_d(\alpha)< 0\}$.
Thus $\Phi(\lambda,\pi,\alpha^*)\ge 0$ 
and  $\alpha^*\leq \alpha_{RS}$. 

For a fixed $\alpha$, we need to minimize $\Phi(\lambda,\pi,\alpha)$ and an easy computation leads to the choice of $\pi$ and $\lambda$ given by:
\BEAS
\pi=\lambda(1-\pi)^d \text{\quad \quad} \alpha = \frac{\lambda(1-\pi)^d}{1+\lambda(1-\pi)^d}
\EEAS
thus we get $\pi=\frac{\alpha}{1-\alpha}$ and
for $H(\alpha) = -\alpha \log(\alpha) - (1-\alpha)\log(1-\alpha)$
\BEA
\label{def:phi_rs}
\Phi_d(\alpha)
&=& H(\alpha) - d\left(\frac{1}{2}(1-2\alpha)\log(1-2\alpha)-(1-\alpha)\log(1-\alpha)\right),
\EEA
which is exactly the expression appearing in a first moment computation, see Lemma 2.1 in \cite{ding2016maximum}.

\subsection{The 1-step of Replica Symmetry Breaking Bound (1-RSB)}
\label{sec_1rsb}

We denote by $\mathcal{L}_1$ the set of probability measures on $\R$, and $\mathcal{L}_2$ the set of probability measures on $\mathcal{L}_1$. 
We will obtain a bound depending on the parameters $m\in (0,1)$ and $\zeta^{(2)}\in \mathcal{L}_2$. We consider the couple of random variables $(\zeta^{(1)},x)$ with the following properties. The random variable $\zeta^{(1)}$ is in $\mathcal{L}_1$ distributed according to $\zeta^{(2)}$. Conditionally on $\zeta^{(1)}$, the real random variable $x$ is distributed according to $\zeta^{(1)}$. We consider i.i.d. copies $(\zeta^{p,(1)}_{i,\ell},x^p_{i,\ell})_{i,\ell,p\in \N}$ of $(\zeta^{(1)},x)$. We define for $\pp$ and $i\geq 1$,
\BEA
\label{eq:defUi-1rsb}U_{p,i}(\sigma;\zeta^{(2)}) = U_p(\theta_{p,i}, x^p_{i,1},\dots, x^p_{i,p-1},\sigma).
\EEA
Note that we are slightly abusing notation here. The definition above is similar to \eqref{eq:defUi} but $\zeta^{(2)}$ is now a distribution on $\mathcal{L}_1$. As a result, the $x^p_{i,\ell}$ are still i.i.d. but with an extra level of randomness as described above. This extra level of randomness is important in our 1-RSB bound given below.

\begin{theorem}\label{th:1rsb}
  If conditions (\ref{cond1},\ref{cond2},\ref{cond3}) and (\ref{cond:deg1},\ref{cond:deg2},\ref{cond:deg3},\ref{cond:deg4}) are satisfied, then for any $m\in (0,1)$ and $\zeta^{(2)}\in \mathcal{L}_2$, we have
  \begin{multline}
  \label{eq:th:1rsb}
F_N \le \frac{1}{m}\mathbb{E}\left[\log\mathbb{E'}\left[\left(\sum_{\sigma=\pm 1}\exp\left( \sum_{i=1}^d U_{p_i,i}(\sigma;\zeta^{(2)}) + h(\sigma) \right)\right)^m \right]\right] \\
    -\frac{\mathbb{E}[d]}{m}\mathbb{E}\left[\frac{p_1-1}{p_1}\log\mathbb{E'}\left(\langle \mathcal{E}_{p_1}\rangle_x\right)^m\right]
    + o_N(1),
    \end{multline}
where $\mathbb{E'}$ is the expectation with respect to $(x_l)$ and $(x^{p}_{i,\ell})$ for fixed $(\zeta^{(1)}_l)$ and $(\zeta^{p,(1)}_{i,\ell})$ and $\mathbb{E}$ denotes the expectation with respect to $(\zeta^{(1)}_l)$, $(\zeta^{p,(1)}_{i,\ell})$, 
$d$ with law $\mu$, $(p_i)_{i\ge 1}$ sequence of i.i.d. random variables with law $\rho$ and the random functions $h$, $(\theta_{p,i})$.
\end{theorem}

Note that if we put all the randomness in only one of the two levels of recursion, Theorem \ref{th:1rsb} reduces to Theorem \ref{thmrs}.
If $\zeta^{(2)}$ has only Dirac measures in its support, then there is no randomness in the second level: $\mathbb{E}'$ vanishes, the parameter
$m$ is cancelled and (\ref{eq:th:1rsb}) becomes (\ref{eq:thmrs}).
On the opposite, if we take $\zeta^{(2)}$ a Dirac mass concentrated on $\zeta\in \mathcal{L}_1$ there is no randomness in the first
level and when $m\to 0$ we get (\ref{eq:thmrs}).
In particular, this bound is a priori tighter than the replica symmetric bound. We demonstrate it on the hard-core model.
\newline

{\bf Application to the hard-core model on $d$-regular graphs:}\\
 We consider the 1-RSB bound for the hard-core model. The mapping $x\mapsto \frac{1}{1+e^{-2x}}$ maps $\R$ to $(0,1)$, so that with the same change of variable as above and making $N$ tend
 to $+\infty$ again, we get:
  \begin{multline*}
\limsup_{N\to\infty}  F_N \le \frac{1}{m}\mathbb{E}\log\mathbb{E'}\left[ \left(1+\lambda\prod_{i=1}^d(1-(1-e^{-A})\pi_i)\right)^{m}\right] \\
    - \frac{d}{2m}\mathbb{E}\log\mathbb{E'}\left[ (1-(1-e^{-A})\pi_1\pi_2)^{m} \right],
 \end{multline*}
 where the $\pi_i$'s are now independent random variables in $(0,1)$ with a random distribution $\eta^{(1)}_i$ where the $(\eta^{(1)}_i)$ are i.i.d. with law $\eta^{(2)}\in \mathcal{L}_2((0,1))$, where $\mathcal{L}_1((0,1))$ is the set of probability measures on $(0,1)$ and $\mathcal{L}_2((0,1))$ the set of probability measures on $\mathcal{L}_1((0,1))$.

\begin{table}
\center
\begin{tabular}{|r|c|c|c|c|c|c|c|c|}
\hline
$d$ & 3 & 4 & 5 & 6 & 7 & 8 & 9 & 10 \\
\hline
$\alpha_{RS}$ & 0.45907 & 0.42061 & 0.38868 & 0.36203 & 0.33944 & 0.32002 & 0.30310 & 0.28820  \\
\hline
$\alpha^{(1)}$ & 0.45086 & 0.41120 & 0.37927 & 0.35299 & 0.33089 & 0.31198 & 0.29556 & 0.28113  \\
\hline
$\alpha_u(r)$ & 0.45537 & 0.41635 & 0.38443 & 0.35799 & 0.33567& &&\\
\hline
$\alpha_\ell(r)$& 0.437575& 0.39213 & 0.35930& 0.33296 & 0.31068 &&&\\
\hline
\end{tabular}
\caption{Numerical values for $\alpha_{RS}>\alpha^{(1)}>\alpha^*$ which are upper bounds for the size of a maximum independent set in a random $d$-regular graphs. For comparison, previous known upper bounds $\alpha_u(r)$ \cite{mckay1987independent} as well as lower bounds $\alpha_\ell(r)$ \cite{hoppen2013local} are provided.}
\label{table:alpha}
\end{table}

We now give a particular choice for $\eta^{(2)}$ that will lead to an improvement on the RS bound. Let $q\in [0,1]$ be the probability that $\pi=1-\frac{1}{\lambda}$ and $1-q$ is the probability that $\pi=\frac{1}{\lambda}$, then $\eta^{(2)}$ (hence $\zeta^{(2)}$)  is chosen
to be trivial and concentrated on this measure.
With this choice, $\mathbb{E}$ vanishes (as explained above) and we will let $m\to 0$ as well as $\lambda\to \infty$ in order to get a better bound than the RS bound.
It turns out that a trivial measure $\zeta^{(2)}$ allows us to improve on the RS bound because we change it as we vary $\lambda$ consistently with $m$, and computations are made possible by the constant degree in the graph.
We now explain the next steps of the computation. First, the terms inside the expectations are bounded and we can use the dominated convergence theorem to make $A$ tend to $+\infty$. We also 
define $\beta$ by  $\log\beta = m\log\lambda$ and we get by (\ref{eq:approxhc}):
\BEAS
\alpha^*\log \beta \leq \log\mathbb{E'}\left[ \left(1+\lambda\prod_{i=1}^d(1-\pi_i)\right)^{m}\right]
    - \frac{d}{2}\log\mathbb{E'}\left[ (1-\pi_1\pi_2)^{m} \right]
\EEAS
where we can compute
 \BEAS
\mathbb{E'}\left[ \left(1+\lambda\prod_{i=1}^d(1-\pi_i)\right)^{m}\right] &=& \sum_{n=0}^d {d \choose n}q^n(1-q)^{d-n}\left(1+\lambda^{1-n}\left(1-\frac{1}{\lambda}\right)^{d-n} \right)^m
\EEAS
and
\begin{multline*}
\mathbb{E'}\left[ (1-\pi_1\pi_2)^{m} \right] = (1-q)^2\left(1-\frac{1}{\lambda^2}\right)^m \\
+2q(1-q)\left(1-\frac{\lambda-1}{\lambda^2}\right)^m +q^2\left( 1-\frac{(\lambda-1)^2}{\lambda^2}\right)^m.
\end{multline*}
Then by taking the limit $m\to 0$, $\lambda\to \infty$ in such a way that $m\log\lambda = \log \beta$, we obtain with:
\BEAS
\Phi^{1}(\beta,q,\alpha):=\log\left(1+(\beta-1)(1-q)^d\right) -\frac{d}{2}\log\left(1-q^2\left(1-\frac{1}{\beta}\right) \right)-\alpha\log \beta,
\EEAS
$\Phi^{(1)}(\beta,q,\alpha^*)\geq 0$ for all $\beta\geq 1$ and $q\in [0,1]$. Hence we define $\Phi^{(1)}(\alpha)=\inf_{q\in [0,1],\beta\geq 1}\Phi^{(1)}(\beta,q,\alpha)$ and $\alpha^{(1)} = \inf\{\alpha>0, \:\Phi^{(1)}(\alpha)<0\}$.
Minimizing in $q$ the function $\Phi^{(1)}(\beta,q,\alpha)$, we find that the optimal value for $q$ is the unique solution in
$[0,1]$ of the equation:
\BEAS
(\beta-1)(1-q)^d + (1-q)^{d-1} + (1-q) -1 = 0 \text{\quad i.e. \quad} \beta = \frac{q}{(1-q)^d}-\frac{q}{1-q}
\EEAS
Hence we can find an expression for $\Phi^{(1)}(\beta,q,\alpha)$ involving only $\alpha$ and $q$.
We did the numerical computations of $\alpha^{(1)}$ and $\alpha_{RS}$ (see Table \ref{table:alpha}).
Note that these values were already computed in \cite{hcrevisited} but we now have a proof that these values are rigorous upper bounds on 
$\alpha^*$ the size of a maximum independent set in the random $d$-regular graph. 
To the best of our knowledge, the best upper bounds on $\alpha^*$ for small degrees were derived by McKay in \cite{mckay1987independent}. 
These values $\alpha_u(r)$ are provided in Table \ref{table:alpha} as well as the lower bounds $\alpha_\ell(r)$ obtained by 
Hoppen and Wormald in \cite{hoppen2013local}.

\subsection{The r-step of Replica Symmetry Breaking Bound (r-RSB)}

For an integer $r\ge 1$, let $0<m_1<\ldots<m_r<1$ be some real parameters.
Let $\mathcal{L}_1$ be a set of probability measures on $\R$, and
by induction for $l\leq r$ we define $\mathcal{L}_{l+1}$ as a set
of probability measures on $\mathcal{L}_{l}.$
Let us fix $\zeta^{(r+1)} \in \mathcal{L}_{r+1}$ (our basic parameter, which is not random) and define a random sequence 
$(\zeta^{(r)}, \zeta^{(r-1)}, \ldots, \zeta^{(1)}, x)$ as follows.  
For $1\leq \ell\leq r+1$, conditionally on $(\zeta^{(r+1)},\dots,\zeta^{(\ell)})$,  $\zeta^{(\ell-1)}$ is an element of $\mathcal{L}_{\ell-1}$ distributed like $\zeta^{(\ell)}$.  
And conditionally on $(\zeta^{(r)},\dots,\zeta^{(1)})$, $x$ is a real random variable with distribution $\zeta^{(1)}$.

For $0\leq j\leq r-1$, we define $\Fcal_j$ the $\sigma$-algebra generated by 
$d$, $(p_i)_{i\ge 0}$, $h$, $(\theta_{p,i})$, 
$\zeta^{(r)}, \zeta^{(r-1)}, \ldots, \zeta^{(r-j)}$ , 
and we denote $\EE_j$ the expectation given $\Fcal_j$.
For a random variable $W \geq 0$ we define $T_rW = W$ and by induction, for $0 \leq l<r$ 
we define the random variable
$T_l W $  by
\begin{equation}
T_l W = \Bigl(
\EE_{l} (T_{l+1} W)^{m_{l+1}}
\Bigr)^{1/m_{l+1}}.
\label{T}
\end{equation}

\begin{theorem}
\label{thmstep}
    If conditions (\ref{cond1},\ref{cond2},\ref{cond3}) and (\ref{cond:deg1},\ref{cond:deg2},\ref{cond:deg3},\ref{cond:deg4}) are satisfied, then for any distribution $\zeta\in\mathcal{L}_{r+1}$, we have
\begin{multline}
F_N \le 
\mathbb{E}\log T_0\left(\sum_{\sigma=\pm 1}\exp\left( \sum_{i=1}^d 
        U_{p_i,i}(\sigma ; \zeta) + h(\sigma) \right)\right) \\
    -\mathbb{E}[d]\mathbb{E}\left[\frac{p_1-1}{p_1}\log T_0 \langle \mathcal{E}_{p_1}\rangle_{x}\right]
    + o_N(1)
\end{multline}
where $d$ is a random variable of law $\mu$,
$U_{p_i,i}(\sigma ; \zeta)$ is defined as in Section \ref{sec_1rsb} but for our new $x$ with $\zeta\in\mathcal{L}_{r+1}$, and
 $(p_i)_{i\ge 1}$ is a sequence of i.i.d. random variables  of law $\rho$. 
\end{theorem}

\section{Proof of Theorem \ref{thmrs}}
\label{sec:proof}

\begin{figure}
\center\includegraphics[width=.8\textwidth]{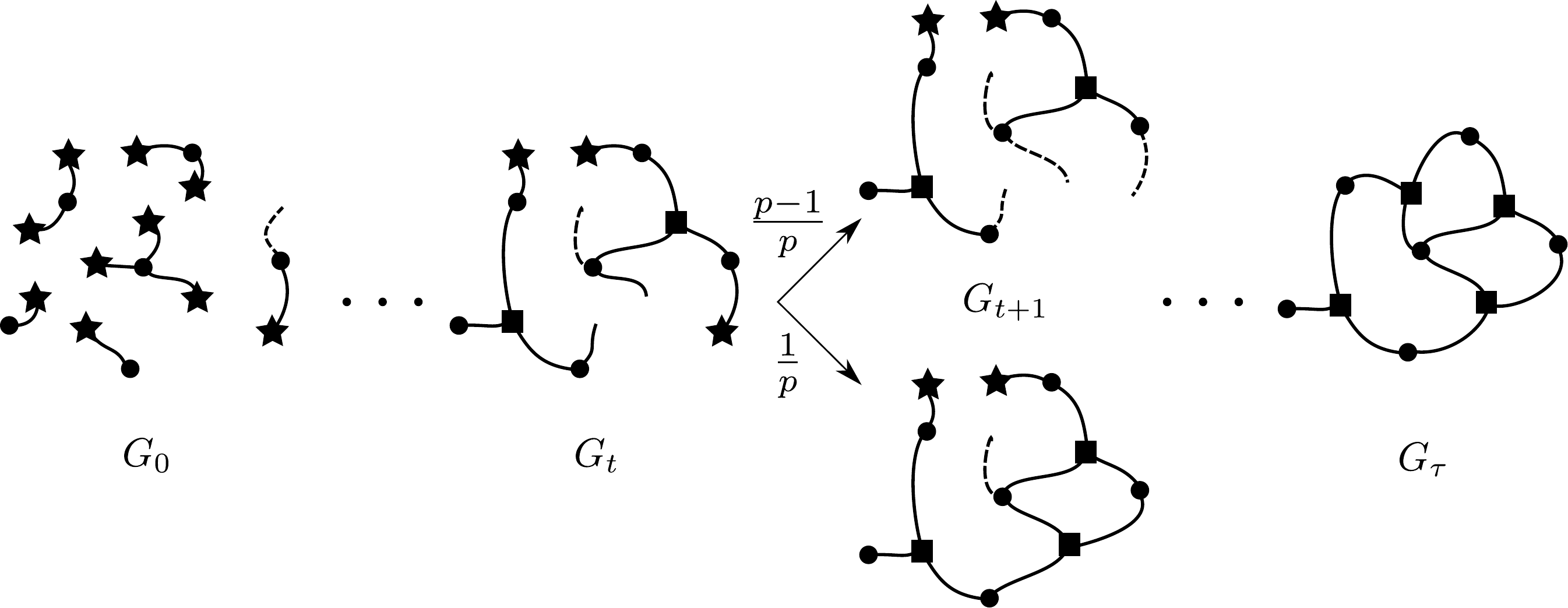}
\caption{Combinatorial interpolation on a random graph by a random walk for $p=3$. Sites are represented by 
stars, and hyperedges by squares.}
\label{fig:interpol}
\end{figure}

Let us first give an overview of the structure of the proof. The fact that we consider many $\pp$ adds generality, but for clarity,
we can assume that all edges have some fixed size $p$.
In Section \ref{subsec:extension}, we explain how we sample
 random graphs, and we extend the definition of graphs
and Hamiltonian by allowing vertices not only to be linked to hyperedges $e$ but also to sites $s$, which are just hyperedges
of cardinal 1. In the Hamiltionian, as we sum $\theta_e$'s over edges, we will sum $U_s$'s over sites. $U_s$ can be seen as the
effect of an external field on the spin linked to $s$. The free energy of a system with only sites and no edges is trivial to compute
since all the spins are decoupled. To bound the free energy associated to our original Hamiltonian, we will interpolate from a
system with sites only to a system with edges only. If the interpolation were
deterministic, we would want to remove $p$ sites and add 1 edge at every step to keep the vertex degrees unchanged. 
But to be able to control the free energy gap,
we will use in \ref{subsec:interpol} a stochastic procedure: at each step, we remove one site chosen uniformly, and with probability 
$\frac{1}{p}$ we add an edge chosen uniformly among all the edges that can be added while respecting the degree constraint. With
probability $1-\frac{1}{p}$ we add nothing. This procedure is adapted from \cite{justin} which uses it in a simpler framework.
Figure \ref{fig:interpol} illustrates this interpolation.
To control the gap at each step, we will need some inequality linking $\theta$ and $U$ that is proved in Section \ref{subsec:replica} 
using the replica method, following the lines of \cite{diluted}. 
Section \ref{subsec:approx} is devoted to the analysis of limits and error terms.

\subsection{Extension of the Graph Model and Matchings}
\label{subsec:extension}

To prove our result, we will need to extend our definition of graphs by allowing them to have (hyper-)edges as above and sites that we will denote by $\mathbf{S}_p$ for $\pp$.
More precisely, we consider graphs of the form $G=(\mathbf{V} , (\mathbf{E}_p)_{p\in\mathbf{P}},(\mathbf{S}_p)_{p\in\mathbf{P}})$ where 
as before $\mathbf{V} = \llbracket 1, N \rrbracket$ for some $N$ 
is a set of vertices, and for any $p\in\mathbf{P}$, $\mathbf{E}_p$ is a set of $p$-edges. Additionally, for each $\pp$, $\mathbf{S}_p$ is a set of $1$-edges, i.e. each $s\in\mathbf{S}_p$ contains exactly one vertex in $\mathbf{V}$ that will be denoted by $\d s$. We need to define the Hamiltonian on this new graph $G$. It will be the same as before plus a contribution for each site, this contribution depending on the type 
(i.e. some $\pp$) of the site.
Namely, with a slight abuse of notation, letting $s$ be the $i$-th element in $\mathbf{S}_p$:
\BEAS
U_{p,s}(\sigma_{\d s};\zeta) = U_p(\theta_{p,i}, x^p_{i,1},\dots, x^p_{i,p-1},\sigma_{\d s}) = U_{p,i}(\sigma_{\d s};\zeta) \mbox{ as defined in (\ref{eq:defUi}).}
\EEAS
For graphs $G=(\mathbf{V} , (\mathbf{E}_p)_{p\in\mathbf{P}},(\mathbf{S}_p)_{p\in\mathbf{P}})$, 
we extend the definition of the Hamiltonian as follows:
for $\sigma \in \Sigma_N$ and a given probability distribution $\zeta$ on $\mathbb{R}$,
\BEA
\label{eq:defHG}-H_G(\sigma)= \sum_{\pp} \left( \sum_{e\in \mathbf{E}_p} \theta_{p,e}(\sigma_{\d e}) +\sum_{s\in \mathbf{S}_p} U_{p,s}(\sigma_{\d s};\zeta)\right)
                              + \sum_{i\in \mathbf{V}}h_i(\sigma_i).
\EEA
We can now explain how we sample these graphs by matching half-edges.
We fix the degree sequence $(d_i)_{i\in\bf V}$, the number of $p$-edges $(E_p)_{p\in\bf P}$ and the number of $p$-sites $(S_p)_{p\in \bf P}$. Since we
allow \emph{unpaired half-edges} as we will see, we don't ask for any relation between $(E_p)_p$, $(S_p)_p$ and $(d_i)_i$, in particular Condition (\ref{cond:deg1}) is not required anymore for these sequences. Define

\begin{equation}
\label{defpartial}
 \mathcal{H}=\bigcup_{i\in \mathbf{V}}\{(i,1),\ldots,(i,d_i) \} \text{ the set of half-edges associated to vertices.} 
\end{equation}
For each $\pp$, there are $E_p$ $p$-edges denoted by $e^p_1, \ldots, e^p_{E_p}$, each of them having $p$ half-edges. 
We denote by $(p,\ell,1), \ldots, (p,\ell,p)$ the half-edges of $e^p_\ell$,  so that half-edges associated to hyper-edges are defined by:
\BEAS
\mathcal{I} = \bigcup_{\pp}\bigcup_{1\le \ell\le E_p}\{ (p,\ell,1),\ldots, (p,\ell,p)\}.
\EEAS
For each $\pp$, there are $S_p$ sites of type $p$, each site $s^p_\ell$ having one half-edge that we denote $(p,\ell,0)$, so that half-edges associated to sites are defined by:
\BEAS
\mathcal{J} = \bigcup_{\pp}\bigcup_{1\le \ell\le S_p}\{(p,\ell, 0)\}.
\EEAS

A (partial) matching $\mathfrak{m}$ between $\mathcal{H}$ and $\mathcal{I}\cup \mathcal{J}$ is a matching in the complete bipartite graph with bipartition $\mathcal{H}$ and $\mathcal{I}\cup \mathcal{J}$. When $\sum_{\pp}\left(pE_p+S_p\right)\leq \sum_{i\in V}d_i$, we say that $\mathfrak{m}$ is a complete matching if all vertices in $\mathcal{I}\cup \mathcal{J}$ are covered (note that this notion
is not symmetric with respect to the bipartition). 
In this case, we denote by $\Mcal=\mathcal{M}((d_i)_{i\in \mathbf{V}},(E_p)_\pp, (S_p)_\pp)$ the set of complete matchings and for $\mathfrak{m}\in \Mcal$, by $G[\mathfrak{m}]=(\mathbf{V}, (\mathbf{E}_p)_\pp,(\mathbf{S_p})_\pp)$ the (hyper-)graph defined on the set of vertices $\mathbf{V}$, where the \mbox{(hyper-)edge} $e^p_\ell$ contains the vertices matched to $(p,\ell,1),\ldots,(p,\ell,p)$ and the site $s^p_\ell$ contains the vertex matched to $(p,\ell,0)$.
Remark that the actual degree sequence of $G[\mathfrak{m}]$ is smaller than $(d_i)_{i\in\bf V}$ since some of half-edges in $\mathcal{H}$ are unpaired.

We define the free energy related to these matchings by
\begin{equation}
F((E_p)_\pp, (S_p)_\pp) = \frac{1}{|\mathcal{M}|}
    \sum_{\mathfrak{m}\in\mathcal{M}}
    \log\sum_{\sigma\in\Sigma_N}\exp\left(
    -H_{G[\mathfrak{m}]}(\sigma)
    \right).
\end{equation}
Note that $F$ also depends on $(d_i)_{i\in \mathbf{V}}$, but this sequence will be kept fixed while we will vary $(E_p)_\pp$ and $(S_p)_\pp$ in the proof. In particular, if we take $d^N =(d_i)_{i\in \mathbf{V}}$, then we have $\mathbb{E} F(E^N,0) = NF_N$ defined in (\ref{eq:defFN}) 
(the expectation is here to average over the randomness in the Hamiltonian).

Given a complete matching $\mathfrak{m}\in\mathcal{M}((d_i)_{i\in \mathbf{V}},(E_p)_\pp, (S_p)_\pp)$ and $\pp$, we can create a larger complete matching
$\mathfrak{m}'\in\mathcal{M}((d_i)_{i\in \mathbf{V}},(E_p)_\pp, (S_p)_\pp+1_p)$ (where $1_p$ is the all zero vector in $\mathbb{R}^\pp$ with a one in the $p$-th position)
by adding to $\mathfrak{m}$ a site of type $p$ as follows: pick one half-edge in $\mathcal{H}$ among those not matched in $\mathfrak{m}$, uniformly at random (provided it exists, i.e. $1+\sum_{q\in \mathbf{P}}\left(qE_q+S_q\right)\leq \sum_{i\in V}d_i$) and match it to the new site. We say that $\mathfrak{m}'$ is obtained from $\mathfrak{m}$ by a random $p$-site-pairing. 
Similarly if $p+\sum_{q\in \mathbf{P}}\left(qE_q+S_q\right)\leq \sum_{i\in V}d_i$, we can create a larger complete matching 
$\mathfrak{m}'\in\mathcal{M}((d_i)_{i\in \mathbf{V}},(E_p)_\pp+1_p, (S_p)_\pp)$
by adding to $\mathfrak{m}$ a $p$-edge as follows: 
pick $p$ half-edges in $\mathcal{H}$ among those not matched in $\mathfrak{m}$, uniformly at random and match them to the new $p$-edge. 
We say that $\mathfrak{m}'$ is obtained from
$\mathfrak{m}$ by a random $p$-edge-pairing.

\begin{lemma}
Let $\mathfrak{M}$ be uniformly distributed on $\mathcal{M}((d_i)_{i\in \mathbf{V}},(E_p)_\pp, (S_p)_\pp)$ and $\pp$. Let $d=\sum_i d_i$, $E=\sum_p pE_p$ and $S=\sum_p S_p$. We assume that $d-E-S\geq p$. Conditionally on $\mathfrak{M}$, make a random $p$-site-pairing (resp. $p$-edge-pairing), then the result $\mathfrak{M}'$ is uniformly distibuted on $\mathcal{M}((d_i)_{i\in \mathbf{V}}, (E_p)_\pp$, $ (S_p)_\pp + 1_p)$ (resp. $\mathcal{M}((d_i)_{i\in \mathbf{V}},(E_p)_\pp+1_p, (S_p)_\pp)$).
\end{lemma}
\begin{proof}
Each $\mfrak\in \mathcal{M}((d_i)_{i\in \mathbf{V}},(E_p)_\pp, (S_p)_\pp)$ admits $d-\sum_{\pp}\left(pE_p+S_p\right)$ allowed $p$-site-pairings, each producing a distinct $\mfrak'\in \mathcal{M}((d_i)_{i\in \mathbf{V}},(E_p)_\pp$, $(S_p)_\pp+1_p)$ containing $\mfrak$. 
Hence $\PP(\mathfrak{M}'=\mfrak')$ is proportional to the number of $\mfrak\in \mathcal{M}((d_i)_{i\in \mathbf{V}},(E_p)_\pp, (S_p)_\pp)$ such that $\mfrak\subset \mfrak'$. This number is exactly $1+\sum_{\pp}S_p$, independently of $\mfrak'$. The proof is similar for the $p$-edge pairing.
\end{proof}

For a complete matching $\mfrak$, we define $\langle \cdot\rangle_{\mfrak}$ the Gibbs average on $\Sigma_N$, with respect to the Hamiltonian $H_{G[\mfrak]}$ defined by (\ref{eq:defHG}) on the random graph $G[\mfrak]$ by
$$
\langle \ind_\sigma \rangle_{\mfrak} = \frac{\exp(-H_{G[\mfrak]}(\sigma))}{\sum_{\sigma'\in\Sigma_N} \exp(-H_{G[\mfrak]}(\sigma'))}
$$
As a direct application of the previous lemma, we obtain
\begin{lemma}\label{lem:ineq}
For $\pp$, if $p+\sum_{q\in \mathbf{P}}\left(qE_q+S_q\right)\leq \sum_{i\in V}d_i$, we have:
 \begin{align*}
 \begin{split}
 F((E_p)_\pp+1_p, (S_p)_\pp)-&{}F((E_p)_\pp, (S_p)_\pp) \\ 
   &    = \frac{1}{|\Mcal|}\sum_{\mathfrak{m}\in\mathcal{M}}\EE_e\log\left\langle\exp(\theta_{p,e}(\sigma_{\partial e}))\right\rangle_{\mathfrak{m}} 
 \end{split} \\
 \begin{split}
  F((E_p)_\pp, (S_p)_\pp+1_p)-&{}F((E_p)_\pp, (S_p)_\pp) \\ 
    &    =\frac{1}{|\Mcal|}\sum_{\mathfrak{m}\in\mathcal{M}}\EE_s\log\left\langle\exp(U_{p,s}(\sigma_{\d s};\zeta))\right\rangle_{\mathfrak{m}},
  \end{split}
  \end{align*}
  where $\mathcal{M}=\mathcal{M}((d_i)_{i\in \mathbf{V}},(E_p)_\pp, (S_p)_\pp)$ and $\EE_e$ (resp. $\EE_s$) denotes the expectation with respect to the random choice of $\d e$ in a random $p$-edge-pairing (resp. $\d s$ in a random $p$-site-pairing) as well as the randomness in the Hamiltonian. \qed
\end{lemma}

\subsection{The Replica Method}
\label{subsec:replica}
We now compute an average quantity related to the right-hand term of the previous lemma which will be crucial to our proof.

\begin{proposition}[Step-by-step increment]
    \label{propstep}
Let $\mathfrak{m}\in\mathcal{M}((d_i)_{i\in \mathbf{V}}$, $(E_p)_\pp$, $(S_p)_\pp)$ a complete matching and $\pp$ such that $\sum_i d_i-\sum_q qE_q-\sum_q S_q \geq \delta>p$. Then we have, for a random $p$-edge-pairing $e$ and a random $p$-site-pairing $s$ independent of each other:
\BEAS
    \mathbb{E}
	\left(
	\frac{1}{p}\log\left\langle\exp \theta_{p,e}(\sigma_{\d e}) \right\rangle_{\mfrak} -  
      \log\left\langle \exp U_{p,s}(\sigma_{\d s};\zeta)\right\rangle_{\mfrak}
	\right)
    \le -\frac{p-1}{p} \mathbb{E}\left[\log \langle\mathcal{E}_{p}\rangle_x\right] 
                + \frac{2p\kappa}{\delta-p},
                \EEAS
                where $x$ in $\langle\mathcal{E}_{p}\rangle_x$ defined by (\ref{eq:defEx}) is a random vector with i.i.d. coordinates distributed according to $\zeta$ and the expectation $\mathbb{E}$ is with respect to the random $p$-edge-pairing and $p$-site-pairing as well as the randomness in the functions $\theta_{p,e}, U_{p,s}, \mathcal{E}_{p}$ and $x$.
    \end{proposition}
\begin{proof}
We first deal with the randomness of the $p$-edge-pairing denoted by $\mathbb{E}_e$.
For $i\in\mathbf{V}$, let $c_i$ be the number of half-edges of $\mathcal{H}$ linked to $i$ that are free in $\mfrak$ and $\chi=\sum_{i\in\mathbf{V}}c_i \ge \delta$. Then for any functional $\phi:\{\pm 1\}^p \to \R$, 
   \BEA
\nonumber    \mathbb{E}_e[\phi(\sigma_{\d e})] &=& \sum_{i_1,\ldots,i_{p}\in\mathbf{V}} 
            \mathbb{P}(\d e=(i_1,\ldots,i_{p}))\phi(\sigma_{i_1},\ldots,\sigma_{i_{p}}) \\
\nonumber   &=& \sum_{i_1,\ldots,i_{p}\in\mathbf{V}} 
    \frac{c_{i_1}}{\chi} \times
    \frac{c_{i_2}-1_{i_1=i_2}}{\chi-1} \times \ldots \\ && \quad \quad \quad \quad \quad  \times
    \frac{c_{i_{p}}-1_{i_1=i_{p}}-\ldots-1_{i_{p-1}=i_{p}}}{\chi-p+1} 
            \phi(\sigma_{i_1},\ldots,\sigma_{i_{p}}) \\
 \label{defZ}  &=& \sum_{i_1,\ldots,i_{p}\in\mathbf{V}} 
        \left( \frac{c_{i_1}\ldots c_{i_{p}}}{\chi^{p}} + Z_{i_1,\ldots,i_{p}}\right)
            \phi(\sigma_{i_1},\ldots,\sigma_{i_{p}}) 
\EEA
    where $Z_{i_1\ldots i_{p}}$ depends on $c_{i_1},\ldots, c_{i_{p}}, \chi$, and 
   $ \sum\limits_{i_1,\ldots,i_{p}\in\mathbf{V}}|Z_{i_1\ldots i_{p}}| 
    \le \frac{2p^2}{\chi-p}$ as proved in Lemma
   \ref{boundZ} below.

    Then by the bound of Condition (\ref{cond2}) applied to $\theta_{p,e}$, 
    $$
    \left| \mathbb{E}_e\log\left\langle\exp(\theta_{p,e}(\sigma_{\d e}))\right\rangle_\mfrak 
   - \sum_{i_1,\ldots,i_{p}\in\mathbf{V}} 
        \frac{c_{i_1}\ldots c_{i_{p}}}{\chi^{p}} 
        \log\left\langle\exp(\theta_{p,e}(\sigma_{\d e}))\right\rangle_\mfrak \right|
    \le \frac{2p^2 \kappa}{\chi-p}.
    $$
    In order to obtain our claim, we need to prove the following inequality~:
    \begin{multline}
    \label{fundineq}
    \sum_{(i_1,\ldots, i_p)\in\mathbf{V}^{p}} 
        \frac{c_{i_1}\ldots c_{i_{p}}}{\chi^{p}} 
    \mathbb{E}_0\log\left\langle\exp(\theta_{p,e}(\sigma_{i_1},
                            \ldots,\sigma_{i_{p}}))\right\rangle_\mfrak \\ - 
    p\sum_{i\in\mathbf{V}}\frac{c_i}{\chi}
      \mathbb{E}_0\log\left\langle \exp(U_{p,s}(\sigma_i;\zeta))\right\rangle_\mfrak
    -(1-p) \mathbb{E}_0\log \langle\mathcal{E}_p\rangle_x 
    \le 0,
    \end{multline}
    where $\mathbb{E}_0$ is the expectation with respect to the randomness in the functions $\theta_p,U_p,\Ecal_p$ and $x$, i.e. under assumption (\ref{cond1}), the randomness in $a_p$, $b_p$, $f_{p,1},\ldots, f_{p,p}$ and $x$ a vector with i.i.d. coordinates with distribution $\zeta$ (independent of the rest). Note that $\mathbb{E}_0$ is independent of the randomness in $\langle\cdot\rangle_\mfrak$.

For the rest of the proof, we will omit the index $p$ in $a_p$, $b_p$ and $f_{p,i}$. As in \cite{diluted}, we introduce replicas $\sigma^1,\ldots, \sigma^{\ell},\ldots$ which are independent copies of $\sigma\in \Sigma_N$ with distribution given by the Gibbs distribution with Hamiltonian $H_{G[\mfrak]}$ defined by (\ref{eq:defHG}).

Using Condition (\ref{cond1}), we have:
    \begin{equation*}
    \begin{split}
    \log\left\langle\exp(\theta_{p,e}(\sigma_{i_1},
                            \ldots,\sigma_{i_{p}}))\right\rangle_\mfrak &= 
    \log(a)-\sum_{n=1}^{+\infty}\frac{(-b)^n}{n}
        \left\langle f_1(\sigma_{i_1})\ldots f_p(\sigma_{i_p})\right\rangle^n_\mfrak \\
    & =
    \log(a)-\sum_{n=1}^{+\infty}\frac{(-b)^n}{n}
        \left\langle \prod_{\ell=1}^n f_1(\sigma^\ell_{i_1})\ldots f_p(\sigma^\ell_{i_p})\right\rangle_\mfrak \\
    \end{split}
    \end{equation*}

\noindent
Then we can define
$$
A_{k,n} = \sum_{i\in\mathbf{V}}\frac{c_i}{\chi}\prod_{\ell=1}^n f_k(\sigma_i^\ell) 
\text{\quad and \quad} 
B_n = \mathbb{E}_0A_{k,n}
$$
such that
\BEAS
\mathbb{E}_0\sum_{(i_1,\dots,i_p)\in\mathbf{V}^p}
        \frac{c_{i_1}\ldots c_{i_p}}{\chi^{p}} 
          \left\langle \prod_{\ell=1}^n f_1(\sigma^\ell_{i_1})\ldots f_p(\sigma^\ell_{i_p})\right\rangle_\mfrak &=&
\mathbb{E}_0\left\langle \prod_{k=1}^p A_{k,n} \right\rangle_\mfrak\\
=
\left\langle\mathbb{E}_0 \prod_{k=1}^p A_{k,n} \right\rangle_\mfrak
&=&
\left\langle B_n^p \right\rangle_\mfrak
\EEAS
Hence we proved that:
\begin{multline}
\label{eq:1}\sum_{(i_1,\ldots, i_p)\in\mathbf{V}^{p}} 
        \frac{c_{i_1}\ldots c_{i_{p}}}{\chi^{p}} 
    \mathbb{E}_0\log\left\langle\exp(\theta_{p,e}(\sigma_{i_1},
                            \ldots,\sigma_{i_{p}}))\right\rangle_\mfrak \\
                         =\mathbb{E}_0[\log a]- \sum_{n=1}^\infty\frac{\mathbb{E}_0\left[(-b)^n\right]}{n}\left\langle B_n^p \right\rangle_\mfrak.
\end{multline}

We do a similar analysis for the second term. Namely, we have for $x_1,\dots,x_{l-1}$ i.i.d. with distribution $\zeta$,
\BEAS
\exp U_{p,s}(\sigma_i;\zeta) &=& \left\langle \Ecal_p\right\rangle_x^-(\sigma_i)
= a\left( 1+bf_p(\sigma_i)\prod_{1\le l\le p-1} \frac{\on{Av}f_l(\epsilon)\exp(\epsilon x_l)}{\on{ch} (x_l)}\right).
\EEAS
Hence, we have
\BEAS
\log \langle \exp U_{p,s}(\sigma_i;\zeta)\rangle_\mfrak &=& 
\log a -\sum_{n=1}^\infty \frac{(-b)^n}{n}\left(\left\langle f_p(\sigma_i)\right\rangle_\mfrak \prod_{1\le l\le p-1} \frac{\on{Av}f_l(\epsilon)\exp(\epsilon x_l)}{\on{ch} (x_l)}\right)^n.
\EEAS
Introducing replicas as above and taking expectation with respect to $\mathbb{E}_0$, we have with $C_n=\mathbb{E}_0 \left(\frac{\on{Av}f_l(\epsilon)\exp(\epsilon x_l)}{\on{ch}(x_l)}\right)^n$,
\BEAS
\mathbb{E}_0\log \langle \exp U_{p,s}(\sigma_i;\zeta)\rangle_\mfrak &=& \mathbb{E}_0[\log a]-\sum_{n=1}^\infty \frac{\mathbb{E}_0\left[(-b)^n\right]C_n^{p-1}}{n}\left\langle \mathbb{E}_0\left[f_p(\sigma^1_i)\dots f_p(\sigma^n_i)\right]\right\rangle_\mfrak,
\EEAS
so that, we get
\BEA
\label{eq:2}\sum_{i\in\mathbf{V}}\frac{c_i}{\chi}
      \mathbb{E}_0\log\left\langle \exp(U_{p,s}(\sigma_i ; \zeta))\right\rangle_\mfrak
=
\mathbb{E}_0[\log a ]- \sum_{n=1}^{+\infty} \frac{\mathbb{E}_0\left[(-b)^n\right]}{n}\langle B_n\rangle_\mfrak C_n^{p-1}
\EEA
Finally, in the same manner, we obtain
\BEA
\label{eq:3}\mathbb{E}_0\log \langle\mathcal{E}_{p}\rangle_x = 
\mathbb{E}_0[\log a ]- \sum_{n=1}^{+\infty} \frac{\mathbb{E}_0\left[(-b)^n\right]}{n} C_n^{p}
\EEA
Using (\ref{eq:1}), (\ref{eq:2}) and (\ref{eq:3}), we see that Inequality (\ref{fundineq}) is equivalent to showing
\begin{equation}
- \sum_{n=1}^{+\infty} \frac{\mathbb{E}_0\left[(-b)^n\right]}{n}  
   \left\langle 
    B_n^p 
    -p \langle B_n\rangle_\mfrak (C_n)^{p-1}
    +(p - 1) (C_n)^{p}
    \right\rangle_\mfrak
    \le 0
\end{equation}
Under Condition (\ref{cond3}), we have $p$ even or $B_n, C_n \ge 0$, and the polynomial
$x^p-pxy^{p-1}+(p-1)y^p$ is always non-negative if $p$ is even or $x,y\ge 0$. 
\end{proof}
    \begin{lemma}
    \label{boundZ}
    Let $Z$ to be defined as in Equation (\ref{defZ}), then
    $$
   \sum_{i_1,\ldots,i_{p}\in\mathbf{V}} |Z_{i_1\ldots i_{p}}| 
        \le 2\sum_{k=1}^{p-1}\frac{k}{\chi-k}
        \le \frac{2p^2}{\chi-p}
    $$
    \end{lemma}
    \begin{proof}
    Keeping light notation, let $\delta_k$ be an alias for $1_{i_1=i_k}+\ldots+1_{i_{k-1=i_k}}$ (hence $\delta_1=0$). 
    For any $p$, we want to find some bound $C_p$ such that
    $$
   \sum_{i_1,\ldots,i_{p}\in\mathbf{V}} |Z_{i_1\ldots i_{p}}| 
   = 
   \sum_{i_1,\ldots,i_{p}\in\mathbf{V}} \left|
   \frac{c_{i_1}\ldots c_{i_p}}{\chi^p} - \frac{(c_{i_1}-\delta_1)\ldots(c_{i_p}-\delta_p)}{\chi\ldots(\chi-(p-1))}
   \right| \le C_p
    $$
    We proceed by induction on $p$. It is trivial that $C_1=0$ works for $p=1$. And for $p>1$
 \BEAS
   &&\sum_{i_1,\ldots,i_{p}\in\mathbf{V}} |Z_{i_1\ldots i_{p}}|\le
   \sum_{i_1,\ldots,i_{p}\in\mathbf{V}} 
   \frac{c_{i_p}}{\chi} \left|
   \frac{c_{i_1}\ldots c_{i_{p-1}}}{\chi^{p-1}} - \frac{(c_{i_1}-\delta_1)\ldots(c_{i_{p-1}}-\delta_{p-1})}{\chi\ldots(\chi-(p-2))}
   \right| \\
   && \qquad \qquad \qquad \qquad \qquad \qquad +
   \left|\frac{c_{i_p}-\delta_p}{\chi-(p-1)}-\frac{c_{i_p}}{\chi}
   \right|\frac{(c_{i_1}-\delta_1)\ldots(c_{i_{p-1}}-\delta_{p-1})}{\chi\ldots(\chi-(p-2))}
   \\
   &\le&
   \sum_{i_p\in\mathbf{V}} \frac{c_{i_p}}{\chi} C_{p-1} +
   \sum_{i_1,\ldots,i_{p-1}\in\mathbf{V}} \frac{(c_{i_1}-\delta_1)\ldots(c_{i_{p-1}}-\delta_{p-1})}{\chi\ldots(\chi-(p-2))}
   \left|\sum_{i_p\in\mathbf{V}} \frac{(p-1)c_{i_p}-\chi\delta_p}{\chi(\chi-(p-1))} \right|
   \\
   &\le&
   C_{p-1} +
   \sum_{i_1,\ldots,i_{p-1}\in\mathbf{V}} \frac{(c_{i_1}-\delta_1)\ldots(c_{i_{p-1}}-\delta_{p-1})}{\chi\ldots(\chi-(p-2))}
   \left( \frac{p-1}{\chi-(p-1)} + \frac{p-1}{\chi-(p-1)} \right)
   \\
   &\le&
    C_{p-1} + 2\frac{p-1}{\chi-(p-1)}
\EEAS
    \end{proof}

\subsection{Graph Interpolation by Random Walks}
\label{subsec:interpol}

We now describe the interpolation scheme leading to the proof of Theorem \ref{thmrs}.
Fix for now a set of parameters $((E_p)_\pp, (S_p)_\pp)$.
We will conduct the interpolation coordinate by coordinate. 
Define $\mathbf{Q}\subset\mathbf{P}$ as $\mathbf{Q}=\{\pp ~|~ S_p\ge \max(15,2p^2)\}$,
fix $q\in\mathbf{Q}$ and suppose
$E_q=0$.
Let $(X_k)_{k\in\mathbb{N}}$ be a sequence of i.i.d. random variables with $\mathbb{P}(X_0=1)=1-\mathbb{P}(X_0=0)=\frac{1}{q}$,
$\mathcal{F}_t$ its natural filtration,
and define
\BEA
\label{defwalk}
\forall t \in \llbracket 0, \tau \rrbracket, 
E_q^t = \sum_{k=1}^t X_k, \quad
S_q^t = \tau - t, \quad (E_p^t, S_p^t)=(E_p, S_p) \text{ for } p\neq q
\EEA
where  $\tau = \tau_q = S_q - 2\delta$ for $\delta = \delta_q$ to be fixed later.

We define the walk of occupied sites, $C_t=qE^t_q+S^t_q$ which is a martingale with mean $\tau$, and 
the stopping time
$$T=\inf\{t\ge 0 ~|~ |C_t-\tau| \ge \delta\}=\inf\{t\ge 0 ~|~ |C_t-\mathbb{E}[C_t]| \ge \delta\}.$$
Finally we define the stopped interpolation by~:
$$
I_t = F((E_p^{t\land T})_\pp,(S_p^{t\land T})_\pp) 
$$
Then for  $t\in\llbracket 0,\tau-1 \rrbracket$, we have 
\begin{align*}
\mathbb{E}[I_{t+1}-I_t~|~ \mathcal{F}_t] =& 
	\ind_{T > t}
	\left(\frac{1}{q}F((E^t_p)_p+1_q,(S^t_p)_p-1_q)\right. \\
&  \quad       +\left.\frac{q-1}{q}F((E^t_p)_p, (S^t_p)_p-1_q) - F((E^t_p)_p, (S^t_p)_p)\right) \\
   =&
	\ind_{T > t}
	\left(\frac{1}{q}\left(F((E^t_p)_p+1_q,(S^t_p)_p-1_q)-F((E^t_p)_p, (S^t_p)_p-1_q)\right) \right.\\
&  \quad  -\left. \frac{q}{q} \left(F((E^t_p)_p, (S^t_p)_p)-F((E^t_p)_p, (S^t_p)_p-1_q)\right)\right) 
\end{align*}
From Lemma \ref{lem:ineq} and Proposition \ref{propstep}, we deduce that for $t\in\llbracket 0,\tau-1 \rrbracket$,
\BEA
\label{eq:ineq}\mathbb{E}I_{t+1}-\mathbb{E}I_t \le \mathbb{P}(T>t)\left(
     -\frac{q-1}{q} \mathbb{E}\log \langle\mathcal{E}_q\rangle_x 
                + \frac{2q\kappa}{\delta-q}
\right).
\EEA

    \begin{proposition}[Ends of the walk]
    \label{propend}
Define for $p\neq q$, $E'_p=E_p$, $S'_p=S_p$, and 
$E'_q=\left\lfloor\frac{ S_q}{q}\right\rfloor$, $S'_q=0$.
Then
\begin{equation}
\begin{split}
\mathbb{E}|F((E_p)_{\pp}, (S_p)_\pp)-I_0| &\le 2\kappa \delta \\
\mathbb{E}|F((E'_p)_{\pp}, (S'_p)_\pp)-I_\tau| &\le 
\kappa \left(6 S_q\exp\left(\frac{-\delta^2}{2\tau q^2}\right) + 3\frac{\delta}{q}+1\right) \\
\end{split}
\end{equation}
    \end{proposition}
    \begin{proof}
Note that a consequence of Condition (\ref{cond2}) is that 
$U_p(\theta_p,x_1,\ldots,x_{p-1},\sigma)$ is also bounded by $\kappa$.
Thus we easily deduce the Lipschitz property from Lemma \ref{lem:ineq}~:
\begin{multline}
\label{flip}
|F((E_p)_\pp, (S_p)_\pp)-
F((E'_p)_\pp, (S'_p)_\pp)| \\ \le
 \kappa \sum_\pp |E_p-E'_p|+|S_p-S'_p| 
\end{multline}
Hence, we have
$$
|F((E_p)_\pp, (S_p)_\pp)-I_0| \le \kappa \sum_\pp|E_p-E^0_p| + |S_p-S^0_p| 
\le 2 \kappa \delta
$$
Since $(C_t)$ is a martingale with increments bounded by $q-1\le q$, we have by Azuma-Hoeffding inequality,
\begin{equation}
\label{azuma}
\mathbb{P}(T\le t) \le 2\exp\left(\frac{-\delta^2}{2tq^2}\right)
\end{equation}
Again by the Lipschitz property (\ref{flip}), 
\begin{equation}
\begin{split}
\mathbb{E}|F((E'_p)_\pp,& (S'_p)_\pp)-I_\tau| \le 
    \kappa \mathbb{E}\left(|E'_q-E^{\tau \land T}_q|+|S'_q-S^{\tau\land T}_q| \right) \\
&\le \kappa  \mathbb{E} \left( \left|E^{\tau \land T}_q - \left\lfloor\frac{S_q}{q}\right\rfloor \right |
    + (\tau-T)\ind_{T<\tau} \right) \\
&\le \kappa \mathbb{E}\left(\ind_{T<\tau}S_q\left(1+\frac{1}{q}\right)  + \ind_{T\ge \tau}\left(3\frac{\delta}{q}+1\right) + \ind_{T<\tau} S_q\right) \\
&\le \kappa \left(S_q\frac{2q+1}{q} 2\exp\left(\frac{-\delta^2}{2\tau q^2}\right) + 3\frac{\delta}{q}+1\right) 
    \text{ and } \frac{2q+1}{q} \le 3
\end{split}
\end{equation} 
    \end{proof}

We are now ready to finish the proof of Theorem \ref{thmrs}. Using Proposition \ref{propend}
 and adding the inequalities (\ref{eq:ineq}) for $t=0\ldots \tau$ we have~:
\begin{equation}
\mathbb{E}F((E'_p), (S'_p))-
    \mathbb{E}F((E_p), (S_p)) \le - S_q
        \frac{q-1}{q} \mathbb{E}\log \langle\mathcal{E}_q\rangle_x 
        + \Delta_q
\end{equation}
where the error term is
\BEAS
\Delta_q&=&2\kappa \delta +
 \kappa \left(6 S_q\exp\left(\frac{-\delta^2}{2\tau q^2}\right) + 3\frac{\delta}{q}+1\right) \\
 &&+ \sum_{t=0}^{\tau-1}
     -\mathbb{P}(T\le t)\frac{q-1}{q} \mathbb{E}\log \langle\mathcal{E}_q\rangle_x 
                + \mathbb{P}(T>t) \frac{2 q\kappa}{\delta-q}  \\
 &&+ (S_q-\tau)\frac{q-1}{q} \mathbb{E}\log \langle\mathcal{E}_q\rangle_x 
\EEAS
By using Condition (\ref{cond2}) on $\log\langle\mathcal{E}_q\rangle_x$ and Equation (\ref{azuma}), we find
\BEAS
|\Delta_q|&\le& 
 \kappa \left(
    7 S_q \exp\left(\frac{-\delta^2}{2\tau q^2}\right) + 3\frac{\delta}{q}+1
        +2\delta
        + \frac{2 q \tau}{\delta-q}
    \right) \\
&\le&
 \kappa \left(
    7 S_q \exp\left(\frac{-\delta^2}{2 S_q q^2}\right) + 3\frac{\delta}{q}+1
        +2\delta
        + \frac{2 q S_q}{\delta-q}
    \right) 
    \EEAS
Remark that if we take $\delta=\left\lceil\sqrt{S_q\log S_q}~\right\rceil$, 
    then since $q\in\mathbf{Q}$, it is easy to check that 
    $S_q > 2\delta$ (from the condition $S_q\ge 15$) and that
    there exists a universal constant $C$ such that~:
    \begin{equation}
   \label{bounddp}
    \frac{|\Delta_q|}{S_q} \le C \text{\quad and \quad}
 \lim_{S_q\rightarrow +\infty}
    \frac{|\Delta_q|}{S_q}=0.
     \end{equation}


\noindent
Now, we can remember the definition of $F_N$ in 
Equation (\ref{eq:defFN}), and apply the previous result coordinate by coordinate for
$p\in\mathbf{Q}$. In the next calculations, given the sequence $(E^N_p)_\pp$,
we define the sets of parameters $(\widetilde{E}_p, \widetilde{S}_p)_\pp$ where
$\forall \pp, \widetilde{E}_p=0, \widetilde{S}_p=pE^N_p$.
\begin{multline}
\label{eqineq}
NF_N = \mathbb{E}F(E^N, 0) 
     \le 
    \mathbb{E}F((\widetilde{E}_p), (\widetilde{S}_p)) - 
    \sum_{\pp}
    \widetilde{S}_p \frac{p-1}{p} \mathbb{E}\log \langle\mathcal{E}_p\rangle_x  \\
        + 
    \sum_{p\in\mathbf{Q}}
        |\Delta_p| +
     \kappa \sum_{p\in \mathbf{P}\setminus\mathbf{Q}} \widetilde{S}_p\left(\frac{p-1}{p}+1+\frac{1}{p}\right)
\end{multline}
where the last term of the right hand side comes from bounding 
$|\mathbb{E}\log \langle\mathcal{E}_p\rangle_x|\le\kappa$
and using the Lipschitz condition (\ref{flip}) between $(E^N,0)$ and $((\widetilde{E}_p), (\widetilde{S}_p))$   
since we did not conduct the interpolation
on $p\in\mathbf{P}\setminus\mathbf{Q}$.

\subsection{Asymptotic Approximation}
\label{subsec:approx}

Until the end of this proof, we will work on the right hand side of Inequality (\ref{eqineq}) by successive
approximations to make Equation (\ref{eq:thmrs}) of Theorem \ref{thmrs} appear.
Writing $\mathcal{M}=\mathcal{M}((\widetilde{E}_j), (\widetilde{S}_j))$, we have, for $i_{p,l}$ being the vertex
matched to the half-edge $(p,l,0)$ associated to a site $s^p_l$, and conversly, $p_{i,d}$ being the type
(that is some $\pp$) of the site matched to the half-edge $(i,d)\in\mathcal{H}$~:
\begin{multline}
\label{separU}
   F((\widetilde{E}_p), (\widetilde{S}_p)) = \\  
   \frac{1}{|\mathcal{M}|}\sum_{\mathfrak{m}\in\mathcal{M}}\log\left(
    \sum_{\sigma\in\{-1,1\}^\mathbf{V}}\exp\left(
        \sum_\pp\sum_{k=1}^{\widetilde{S}_p} U_{p,k}(\sigma_{i_{p,k}} ; \zeta)
            + \sum_{i\in\mathbf{V}} h_i(\sigma_i)
        \right)
   \right) \\
   =
   \frac{1}{|\mathcal{M}|}\sum_{\mathfrak{m}\in\mathcal{M}}\log\left(
   \prod_{i\in\mathbf{V}}\sum_{\sigma=\pm 1}
        \exp\left(\sum_{d=1}^{d_i}U_{p_{i,d},(i,d)}(\sigma ; \zeta) + h_i(\sigma)
            \right)
   \right) \\
   =
   \sum_{i\in\mathbf{V}}\frac{1}{|\mathcal{M}|}\sum_{\mathfrak{m}\in\mathcal{M}}
   \log\left(
        \sum_{\sigma=\pm 1}
        \exp\left(\sum_{d=1}^{d_i}U_{p_{i,d}, (i,d)}(\sigma ; \zeta) + h_i(\sigma)
            \right)
   \right) \\
\end{multline}
In the following, according to Condition \ref{cond:deg1}, we define the quantity
$M=\sum_{i=1}^N d_i^N = \sum_{\pp}pE_p^N$.
Given $i\in\mathbf{V}$, $p_1,\ldots,p_{d_i}\in\mathbf{P}$ 
    and under a uniform choice of $\mathfrak{m}\in\mathcal{M}$, we have that
$\mathbb{P}_{\mathfrak{m}}(p_{i,1}=p_1,\ldots,p_{i,d_i}=p_{d_i})$ equals to
$$
    \frac{\widetilde{S}_{p_1}}{M} 
        \frac{\widetilde{S}_{p_2}-1_{p_1=p_2}}{M-1} \ldots 
        \frac{\widetilde{S}_{p_{d_i}}-1_{p_1=p_{d_i}}\ldots - 1_{p_{d_i-1}=p_{d_i}}}{M-d_i+1}
$$
Hence by Lemma \ref{boundZ} and Condition (\ref{cond2}), we will approximate
$F((\widetilde{E}_p), (\widetilde{S}_p))$ by the following term~:
$\widehat{F}((\widetilde{E}_p), (\widetilde{S}_p)) \coloneqq $
$$
    \sum_{i\in\mathbf{V}}\sum_{p_1,\ldots,p_{d_i} \in \mathbf{P}}
        \frac{\widetilde{S}_{p_1}\ldots \widetilde{S}_{p_{d_i}}}{M^{d_i}}
   \log\left(
        \sum_{\sigma=\pm 1}
        \exp\left(\sum_{d=1}^{d_i}U_{p_d,(i,d)}(\sigma ; \zeta) + h_i(\sigma)
            \right)
        \right)
$$ 
\begin{multline}
\label{eqecart}
\left|F((\widetilde{E}_p), (\widetilde{S}_p)) - \widehat{F}((\widetilde{E}_p), (\widetilde{S}_p))  
    \right| \\
\le
    \sum_{i\in\mathbf{V}}\sum_{p_1,\ldots,p_{d_i}\in\mathbf{P}}
        \left|
            \frac{\widetilde{S}_{p_1}\ldots \widetilde{S}_{p_{d_i}}}{M^{d_i}}
            - 
        \mathbb{P}_{\mathfrak{m}}(p_{i,1}=p_1,\ldots,p_{i,d_i}=p_{d_i}) 
         \right| (\log(2)+(d_i+1)\kappa)    \\
\le
    \sum_{i\in\mathbf{V}}(\log(2)+(d_i+1)\kappa)\frac{2d_i^2}{M-d_i}
\end{multline}
Putting Equations (\ref{eqineq}), (\ref{separU}) 
and (\ref{eqecart}) together, we find~:
\begin{multline}
NF_N \le
\mathbb{E}\widehat{F}((\widetilde{E}_p), (\widetilde{S}_p))
-
    \sum_{\pp}
    \widetilde{S}_p \frac{p-1}{p} \mathbb{E}\log \langle\mathcal{E}_p\rangle_x  \\
        + 
    \sum_{p\in\mathbf{Q}}
        |\Delta_p| +
     2\kappa \sum_{p\in \mathbf{P}\setminus\mathbf{Q}} \widetilde{S}_p
+
    \sum_{i\in\mathbf{V}}(\log(2)+(d_i+1)\kappa)\frac{2d_i^2}{M-d_i}
\end{multline}

It remains to study the limits of the terms of the right-hand side as $N$ tends to infinity. 
Consistently with the definitions of the measures $\mu, \nu$ and $\rho$, recall the definitions of $\mu_N, \nu_N$
(\ref{cond:deg2},\ref{cond:deg3}) and we define the empirical measure $\rho_N$:
$$
\forall p\in\mathbf{P}, 
 \rho_N(p)=\frac{pE^N_p}{\sum_{q\in\mathbf{P}}qE^N_q}
$$
Moreover, for any probability measure, say $f$, on $\mathbb{N}$, we write its mean $\bar{f}=\sum_{n\in\mathbb{N}}nf(n)$.
\newline

\textbf{First Term.}
With the definitions in the statement of Theorem \ref{thmrs},
\begin{equation}
\begin{split}
&\left|\frac{1}{N}\mathbb{E}\widehat{F}((\widetilde{E}_j), (\widetilde{S}_j)) -
    \mathbb{E}\log\left(\sum_{\sigma=\pm 1}\exp\left(\sum_{k=1}^dU_{p_k,k}(\sigma ; \zeta) + h(\sigma)\right)\right) \right| \\
= &   \left|\sum_{d\in\mathbb{N}}\sum_{p_1,\ldots,p_d}
        \left(\mu_N(d)\rho_N(p_1)\ldots\rho_N(p_d)
            - \mu(d)\rho(p_1)\ldots\rho(p_d)\right)  \right. \\ 
 &   \qquad \qquad \qquad  \qquad \qquad  \left.     
        \mathbb{E}\log\left(\sum_{\sigma=\pm 1}\exp\left(\sum_{k=1}^dU_{p_k,k}(\sigma ; \zeta) + h(\sigma)\right)\right) \right| 
            \\
\le&
      \sum_{d\in\mathbb{N}}\sum_{p_1,\ldots,p_d}
        |\mu(d)-\mu_N(d)|\rho_N(p_1)\ldots\rho_N(p_d)
            (\kappa (d+1) +\log(2)) \\
& +
      \sum_{d\in\mathbb{N}}\sum_{p_1,\ldots,p_d}
        \mu(d)
        \left|\rho_N(p_1)\ldots\rho_N(p_d)
        - \rho(p_1)\ldots\rho(p_d)\right|
            (\kappa (d+1) +\log(2))
            \\
  \le &
      \sum_{d\in\mathbb{N}}
        |\mu(d)-\mu_N(d)|
            (\kappa (d+1) +\log(2)) \\
&+
      \sum_{d\in\mathbb{N}}
        \mu(d)
            (\kappa (d+1) +\log(2))
      \sum_{p_1,\ldots,p_d}
        \left|\rho_N(p_1)\ldots\rho_N(p_d)
        - \rho(p_1)\ldots\rho(p_d)\right|
\end{split}
\end{equation}
Hence using Conditions (\ref{cond:deg2},\ref{cond:deg3},\ref{cond:deg4}), it is easy to show (by a direct application for the first sum,
and a version of the dominated convergence theorem for the second one), that

\begin{equation}
\lim_{N\rightarrow\infty} \left|\frac{1}{N}\mathbb{E}\widehat{F}((\widetilde{E}_j), (\widetilde{S}_j)) -
    \mathbb{E}\log\left(\sum_{\sigma=\pm 1}\exp\left(\sum_{k=1}^dU_{p_k,k}(\sigma ; \zeta)+h(\sigma)\right)\right) \right| = 0
\end{equation}

\textbf{Second Term.}
$$
\frac{1}{N} \sum_\pp
    \widetilde{S}_p \frac{p-1}{p} \mathbb{E}\log \langle\mathcal{E}_p\rangle_x
    =
 \frac{M}{N}\sum_{\pp}\rho_N(p)\frac{p-1}{p}\mathbb{E}\log \langle\mathcal{E}_p\rangle_x
$$
By Conditions (\ref{cond:deg2},\ref{cond:deg3},\ref{cond:deg4}), $\frac{M}{N}$ tends to $\bar{\mu}=\mathbb{E} d$, 
and since $\mathbb{E}\log \langle\mathcal{E}_p\rangle_x$ is bounded by $\kappa$ 
we have

\begin{equation}
\lim_{N\rightarrow\infty}
    \frac{1}{N} \sum_\pp
    \widetilde{S}_p \frac{p-1}{p} \mathbb{E}\log \langle\mathcal{E}_p\rangle_x =
    \mathbb{E} [d] \mathbb{E}\left[\frac{p-1}{p}\log\langle \mathcal{E}_p\rangle_x\right] 
\end{equation}

\textbf{Third Term.}
The following term tends to zero 
by Equations (\ref{bounddp}) and (\ref{cond:deg3}), since
$\frac{M}{N} $ converges, $\rho_N$ converges to $\rho$ and
$\frac{\Delta_p}{\widetilde{S}_p}$ is uniformely bounded in $p$ and tends to zero:
\begin{equation}
    \frac{1}{N} \sum_{p\in\mathbf{Q}}
        |\Delta_p|  = \frac{M}{N} \sum_{p\in\mathbf{Q}} \rho_N(p) \frac{\Delta_p}{\widetilde{S}_p} \le
         \frac{M}{N} \sum_{p\in\mathbf{Q}} \rho(p) \frac{\Delta_p}{\widetilde{S}_p} + |\rho(p)-\rho_N(p)|C
\end{equation}

\textbf{Fourth Term.}
Similarly, the next term tends to zero since
$\ind_{p\in\mathbf{P}\setminus\mathbf{Q}}$ is uniformely bounded by 1 and is eventually equal to zero for all $p$ such that
$\rho(p)>0$:
\begin{equation}
\frac{1}{N}\sum_{p\in\mathbf{P}\setminus\mathbf{Q}} \widetilde{S}_p = 
\frac{M}{N}\sum_{p\in\mathbb{N}} \rho_N(p) \ind_{p\in\mathbf{P}\setminus\mathbf{Q}} 
\end{equation}

\textbf{Fifth Term.}
To show that
$ 
\lim_{N\rightarrow +\infty} 
    \frac{1}{N} 
    \sum_{i\in\mathbf{V}}(\log(2)+(d_i+1)\kappa)\frac{2d_i^2}{M-d_i} = 0
    $, it is enough to show that 
    $
    \frac{1}{N}\sum_{i\in\mathbf{V}}\frac{d_i^3}{M-d_i}
    $ tends to zero.
$$
\frac{1}{N}\sum_{i\in\mathbf{V}}\frac{d_i^3}{M-d_i} 
= 
\frac{N}{M}\sum_{i\in\mathbf{V}}\frac{d_i^3}{N^2}\frac{M}{M-d_i} 
= 
\frac{N}{M}\sum_{i\in\mathbf{V}}\frac{d_i^3}{N^2}
+
\frac{N}{M}\sum_{i\in\mathbf{V}}\frac{d_i^3}{N^2}\frac{d_i}{M-d_i} 
$$
Note that $\frac{N}{M}$ converges. By Condition (\ref{cond:deg4}) and Cauchy-Schwartz inequality, 

$$
\frac{1}{N^2}\sum_{i=1}^N d_i^3 \le \frac{1}{N}\sqrt{\frac{1}{N}\sum d_i^2}\sqrt{\frac{1}{N}\sum d_i^4}
\le \frac{1}{N^2}\sqrt{\sum d_i^2} \sum d_i^2 = O(\frac{1}{\sqrt{N}}) 
$$
Hence to finish the proof, it remains to show that $\limsup_{N\rightarrow\infty} \sup_i \frac{d^N_i}{M} < 1$.

\begin{multline}
\text{If } \frac{d_i}{M} \ge \frac{d_i\mu(d_i)}{\bar{\mu}} \text{ then }
\frac{d_i}{M}-\frac{d_i\mu(d_i)}{\bar{\mu}} \le
\left|
\frac{d_i \#\{j\in\mathbf{V}| d_j=d_i\}}{M}-\frac{d_i\mu(d_i)}{\bar{\mu}} 
\right| 
\\
\le
d_i\left|
    \frac{\mu_N(d_i)}{\overline{\mu_N}} - \frac{\mu(d_i)}{\bar{\mu}}
    \right|
    \le
\left|\frac{1}{\overline{\mu_N}}-\frac{1}{\bar{\mu}}\right|\sum_{d\in\mathbb{N}}d\mu(d) + 
    \frac{1}{\overline{\mu_N}}\sum_{d\in\mathbb{N}}d|\mu(d)-\mu_N(d)|
\end{multline}
Condition (\ref{cond:deg2}) shows that the right-hand term tends to zero and it is independent of $i$, 
thus, if $\mu$ is not concentrated in one point,

$$
\limsup_{N\rightarrow\infty} \sup_i \frac{d^N_i}{M} \le \sup_d \frac{d\mu(d)}{\bar{\mu}} < 1
$$
If $\mu$ is concentrated in one point $d$, then $\lim_{N\rightarrow\infty} \frac{d}{M}=0$.
This concludes the proof of Theorem~\ref{thmrs}. \qed

\section{A General Weighted Bound}
\label{sec:weight}

Sections \ref{sec:weight} and \ref{sec:step} describe how to tune and generalize the bound of Theorem \ref{thmrs}. 
In terms of proof techniques, the actual contribution of this paper is Theorem \ref{thmrs}: the generalization is just
a rewriting of the proof of Panchenko and Talagrand that generalizes \cite[Theorem 1]{diluted} to \cite[Theorem 4]{diluted}. 
Instead of copy-pasting
half of \cite{diluted}, we made the choice to refer the reader to the proof of \cite{diluted}, and to just present in the next sections
the small changes to make to the proof of Theorem \ref{thmrs}.

We use the weighting scheme defined in \cite{diluted} with the same notations~: $\Gamma$ is a countable set,
$(x^\gamma)_{\gamma\in\Gamma}$ a sequence of random variables depending somehow (this will be fixed in the next section)
on a distribution $\zeta \in \mathcal{L}_{r+1}$, $(x^{s,\gamma}_l)_{s\in\mathbf{S},l\ge 0}$ are independent copies of this sequence,
and
$(\theta_{p,s})_{s\in\mathbf{S}}$ are independent copies of $\theta_p$. As in definitions (\ref{eq:defUi},\ref{eq:defUi-1rsb}), define
\BEA
U^{\gamma}_{p,s}(\epsilon ; \zeta)=U_p(\theta_{p,s},x^{s,\gamma}_1,\ldots, x^{s,\gamma}_{p-1},\epsilon)
\EEA
and for a graph $G(\mathbf{V},(\mathbf{E}_p)_{\pp}, (\mathbf{S}_p)_{\pp})$,
$\sigma\in\Sigma_N$
\begin{equation}
-H_G^\gamma(\sigma) =  \sum_\pp \left( \sum_{e\in \mathbf{E}_p} \theta_{p,e}(\sigma_{\d e}) 
                              +  \sum_{s\in \mathbf{S}_p} U^{\gamma}_{p,s}(\sigma_{\d s} ; \zeta)  \right)
                                + \sum_{i\in\mathbf{V}}h_i(\sigma_i).
\end{equation}
Note that for $\gamma\in\Gamma, H_{G_N}^\gamma=H_{G_N}=H_N$ since $G_N$ has only edges and no sites.

For a sequence of non-negative random variables $(v_\gamma)_{\gamma\in\Gamma}$ with $\sum_{\gamma\in\Gamma}v_\gamma=1$, we
define the Gibbs measure and its Gibbs average on $\Sigma_N\times\Gamma$ by
\begin{equation}
\label{gaverage}
\langle \ind_{\sigma,\gamma} \rangle_G = v_\gamma\exp(-H^\gamma_G(\sigma))/Z_N
\end{equation}
where $Z_N=\sum_{\sigma,\gamma}v_\gamma\exp(-H^\gamma_G(\sigma))$.

We will need to consider the successive steps of the interpolation defined in
Equation (\ref{defwalk}) along each coordinate. Let
$(E_p)_\pp, (S_p)_\pp$ and
$(\widetilde{E}_p)_\pp$, $(\widetilde{S}_p)_\pp$ 
to be defined as in the proof of Theorem \ref{thmrs} (that is
$\widetilde{G}_N$ is the graph with sites only except for $p\notin \mathbf{Q}$)
and 
$(E^t_p)_\pp, (S^t_p)_\pp$  as in Equation (\ref{defwalk})
where $\tau=\tau_p=\widetilde{S}_p-2\delta_p$, we have
for an enumeration of 
$\{p_1,\ldots,p_k\}$ of the values in $\mathbf{P}$ that appear in the graph (i.e. $E_p+S_p \neq 0$), 
$1 \le t\le \tau_{p_i}$ for some $i$, 
\begin{multline}
G^{p_i,t} = G\left((E_{p_1},\ldots, E_{p_{i-1}}, E^t_{p_i}, \widetilde{E}_{p_{i+1}}\ldots, \widetilde{E}_{p_k}),\right.\\ 
    \left.(S_{p_1}, \ldots, S_{p_{i-1}}, S^t_{p_i}-1, \widetilde{S}_{p_{i+1}},\ldots, \widetilde{S}_{p_k})\right)
\end{multline}
It is the graph defining the average used in
Proposition \ref{propstep} at the $t$-th step of the successive interpolation on the $p_i$-th coordinate.

\begin{theorem}[A general bound]
    \label{thmgeneral}
    Let
$$
\widetilde{F}_N=\frac{1}{N}\mathbb{E}\log\sum_{\sigma\in\Sigma_N,\gamma\in\Gamma}v_\gamma\exp(-H^\gamma_{\widetilde{G}_N}(\sigma))
$$

    If conditions (\ref{cond1},\ref{cond2},\ref{cond3}) are satisfied, then
\begin{equation}
F_N \le \widetilde{F}_N
    - \frac{1}{N}\sum_{p\in \mathbf{Q}}
        \sum_{t=1}^{\tau_p}
        \frac{p-1}{p} \mathbb{E}\log \left\langle \langle\mathcal{E}_p\rangle_{x^\gamma} \right\rangle_{G^{p,t}}
    + o_N(1)
\end{equation}
where $x^\gamma=(x^\gamma_1,\ldots,x^\gamma_{p})$, and $(x^\gamma_l)_{\gamma\in\Gamma}$, for $l\in\mathbb{N}$ are independent copies
of $(x^\gamma)_{\gamma\in\Gamma}$.
    \end{theorem}

\begin{proof}
The proof of Theorem \ref{thmrs}, by interpolating coordinate by coordinate holds with almost no changes except that the free
energy is now defined for the new partition function (\ref{gaverage}).
Proposition \ref{propend} holds with no changes to its proof, and
the equation of Proposition \ref{propstep} becomes
\begin{multline*}
    \mathbb{E}
	\left(
	\frac{1}{p}\log\left\langle\exp(\theta_{p,e}(\sigma_{\d e}
                            ))\right\rangle_{G[\mathfrak{m}]} -  
      \log\left\langle \exp(U^{\gamma}_{p,s}(\sigma_{\d s}; \zeta))\right\rangle_{G[\mathfrak{m}]}
	\right) \\
    \le -\frac{p-1}{p} \mathbb{E}\log \left\langle\langle\mathcal{E}_p\rangle_{x^\gamma}\right\rangle_{G[\mathfrak{m}]}
                + \frac{2p\kappa}{\delta-p}
\end{multline*}
where $\langle\cdot\rangle_{G[\mathfrak{m}]}$ is now defined for some 
matching $\mathfrak{m}\in\mathcal{M}$ by the new average of Equation (\ref{gaverage}).  
As in \cite{diluted}, the rest of the proof is the same, except that we now define
$$
C_n(\gamma_1,\ldots,\gamma_n)=\mathbb{E}_0\prod_{i=1}^n\frac{\on{Av}f_{p,1}(\epsilon)\exp\epsilon x_1^{\gamma_i}}{\on{ch}x_1^{\gamma_i}}
$$
\end{proof}

\section{The r-step of Replica Symmetry Breaking Bound}
\label{sec:step}

We will now make a specific choice for the weights $(v_\gamma)_\gamma$ and
the random variables $(x_\gamma)_\gamma$ to get an explicit bound. Again, we refer the reader 
to \cite{diluted} where this choice is thoroughly explained and proved. 

For the proof, we need to complexify the definitions of the statement of Theorem \ref{thmstep}.
This extra level of complexity will vanish at the end of the proof.
For an integer $r\ge 1$, let $\Gamma=\mathbb{N}^r$ and $0<m_1<\ldots<m_r<1$ be some real parameters.
Let $\mathcal{L}_1$ be a set of probability measures on $\R$, and
by induction for $l\leq r$ we define $\mathcal{L}_{l+1}$ as a set
of probability measures on $\mathcal{L}_{l}.$
Let us fix $\zeta^{(r+1)} \in \mathcal{L}_{r+1}$ (our basic parameter) and define a random sequence 
$(\eta, \eta(\gamma_1), \ldots, \eta(\gamma_1, \ldots, \gamma_{r-1}), x(\gamma_1,\ldots, \gamma_r))$ as follows.  
The element $\eta$ of $\mathcal{L}_r$ is distributed according to $\zeta$. Given $\eta$, 
the sequence $(\eta(\gamma_1))_{\gamma_1 \geq 1}$ of elements of $\mathcal{L}_{r-1}$ is i.i.d distributed like $\eta$.  
For $1\leq l\leq r-1$,  given  all the elements 
$\eta(a_1, \ldots, a_{s})$ for all values of the integers $a_1, \ldots, a_s$ and all $s \leq l-1$,
 the sequence $(\eta(\gamma_1, \ldots, \gamma_l))_{\gamma_l \geq 1}$ of elements of $\mathcal{L}_{r-l}$  
is i.i.d  distributed like $\eta(\gamma_1, \ldots, \gamma_{l-1})$, and these sequences are independent 
of each other for different values of $ (\gamma_1, \ldots, \gamma_{l-1})$.   
Finally, given  all the elements $\eta(a_1, \ldots, a_{s})$ for all values of the integers 
$a_1, \ldots, a_s$ and all $s \leq r-1$ the sequences  
$x(\gamma_1,\ldots,\gamma_r), \gamma_r\geq 1$ 
are an i.i.d. sequences on $\R$ with the distribution
$\eta(\gamma_1,\ldots,\gamma_{r-1})$ and these sequences
are independent for different values of  $(\gamma_1,\ldots,\gamma_{r-1}).$
The process of generating the $x$'s can be represented schematically as
\begin{equation}
\zeta\to\eta \to\eta(\gamma_1)\to\ldots
\to\eta (\gamma_1,\ldots,\gamma_{r-1})\to
x(\gamma_1,\ldots,\gamma_r).
\label{xs}
\end{equation}
Let us consider an arbitrary countable index
set $\Omega$ that will be fixed to $\Omega=\mathbf{P}\times\mathbb{N}^2$.
For $\omega \in \Omega$, we consider independent copies 
$(\eta_\omega, \eta_\omega(\gamma_1), \ldots$, $\eta_\omega(\gamma_1, \ldots, 
\gamma_{r-1})$, $x_\omega(\gamma_1,\ldots, \gamma_r))$ of $(\eta, 
\eta(\gamma_1), \ldots,  $ $\eta(\gamma_1, \ldots, \gamma_{r-1}), 
x(\gamma_1,\ldots, \gamma_r))$.

For $ 0\leq j\leq r-1,$
let us denote by $\Fcal_j$ the $\sigma$-algebra generated by 
$\eta_{\omega}(\gamma_1,\ldots,\gamma_l)$
for $\omega\in\Omega,$ $l\leq j,$ $\gamma_1,\ldots,\gamma_l\geq 1$, and by the random variables $h_i$, 
$\theta_{p,i,j}$. 
Let us denote by $\EE_{j}$ the expectation given $\Fcal_j$ or, in other words,
with respect to $\eta_{\omega}(\gamma_1,\ldots,\gamma_l)$
for $\omega\in\Omega,$ $l>j,$ $\gamma_1,\ldots,\gamma_l\geq 1$
and
$x_{\omega}(\gamma_1,\ldots,\gamma_r)$
for $\omega\in\Omega,$ $\gamma_1,\ldots,\gamma_r\geq 1.$ In particular $\Fcal_0$ is 
generated by the variables $\eta_\omega$, $h_i$, $\theta_{p,i,j}$.

For a random variable $W \geq 0$ we define $T_rW = W$ and by induction, for $0 \leq l<r$ 
we define the random variable
$T_l W $  by
\begin{equation}
T_l W = \Bigl(
\EE_{l} (T_{l+1} W)^{m_{l+1}}
\Bigr)^{1/m_{l+1}}.
\label{T}
\end{equation}
We take $\Gamma=\mathbb{N}^r$, some parameters $0<m_1<\ldots<m_r<1$ and we define
$v_\gamma$ using Derrida-Ruelle cascades \cite{ruelle1987mathematical,panchenko2007guerra}. For $i=1,\ldots,r$,
let $(u_{\gamma_i})_{\gamma_i \ge 1}$ be the non-increasing enumeration of the points
generated by a Poisson point process on $\mathbb{R^+_\star}$ with intensity 
$x\mapsto x^{-1-m_r}$.

Consider a sequence $(u_{\gamma_1,\ldots,\gamma_l})_{\gamma_1,\ldots,\gamma_l\ge 1}$ such 
that for a fixed $(\gamma_1,\ldots,\gamma_{l-1})$, it is an independent copy of $(u_{\gamma_l})_{\gamma_l\ge 1}$.
Then
\begin{equation}
v_{\gamma_1,\ldots,\gamma_r}=\frac{\prod_{l=1}^r u_{\gamma_1,\ldots,\gamma_r}}{
        \sum_{\gamma'_1,\ldots,\gamma'_r}\prod_{l=1}^r u_{\gamma'_1,\ldots,\gamma'_r}}
\end{equation}
We use the following proposition, proved in \cite[Proposition 2]{diluted}.

\begin{proposition}
\label{propcascade}
Consider a function $V:\mathbb{R}^\Omega \mapsto \mathbb{R}, V\ge 0$ and the random variable
defined by $V(\gamma_1,\ldots,\gamma_r)=V((x_\omega(\gamma_1,\ldots,\gamma_r))_{\omega\in\Omega})$.
The random variable $T_l(V(\gamma_1, 
\ldots, \gamma_r))$ does not depend on $\gamma_{l+1}, \ldots, \gamma_r$, in particular 
the law of
$T_0 V(\gamma_1,\ldots,\gamma_r)$ does not depend on $\gamma_1,\ldots,\gamma_r$.
Assume that $\mathbb{E}V(\gamma_1,\ldots,\gamma_r)^2 < \infty$, then for arbitrary values
of the $\gamma_i$'s in the right hand side
\begin{equation}
\mathbb{E}\log\sum_{\gamma_1,\ldots,\gamma_r\ge 1}v_{\gamma_1,\ldots,\gamma_r}V(\gamma_1,\ldots,\gamma_r)=
\mathbb{E}\log T_0 V(\gamma_1,\ldots,\gamma_r) 
\end{equation} \qed
\end{proposition}

\begin{proof}[Proof of Theorem \ref{thmstep}]
For $\pp, 1\le t \le \tau_p$ and $\gamma=(\gamma_1,\ldots,\gamma_{r})$, 
let $e(\gamma)=\langle\mathcal{E}_p\rangle_{x^\gamma}$ and 
    $Z(\gamma)= \sum_{\sigma\in\Sigma_N} \exp(-H^{\gamma}_{G^{p,t}}(\sigma))$. Then
we can apply Proposition \ref{propcascade} to show
\begin{multline}
\label{eq:ung-E}
\mathbb{E}\log\left\langle\langle\mathcal{E}_p\rangle_{x^\gamma} \right\rangle_{G[\mfrak]} 
= 
\mathbb{E}\log \frac{\sum_{\gamma}v_\gamma Z(\gamma)e(\gamma)}{\sum_{\gamma}v_\gamma Z(\gamma) }
\\ =
\mathbb{E}\log \sum_{\gamma}v_\gamma Z(\gamma)e(\gamma) - \mathbb{E}\log\sum_{\gamma}v_\gamma Z(\gamma) 
 =
\mathbb{E}\log T_0(eZ) - \mathbb{E}\log T_0Z 
\\ = 
\mathbb{E}\log (T_0e)(T_0Z) - \mathbb{E}\log T_0Z = 
\mathbb{E}\log T_0e
=    \mathbb{E}\log T_0\langle\mathcal{E}_p\rangle_{x^\gamma} 
\end{multline}
To produce the other term of Theorem \ref{thmstep}, we remark that, using notation of Equation (\ref{separU})~:
\begin{multline}
\label{eq:ung-U}
\widetilde{F}_N=\frac{1}{N}\mathbb{E}\log\sum_{\sigma\in\Sigma_N,\gamma\in\Gamma}v_\gamma\exp(-H^\gamma_{\widetilde{G}_N}(\sigma)) \\
=   \frac{1}{N}
   \frac{1}{|\mathcal{M}|}\sum_{\mathfrak{m}\in\mathcal{M}}
   \mathbb{E}\log\left(
   \sum_{\gamma}v_\gamma
   \prod_{i\in\mathbf{V}}\sum_{\sigma=\pm 1}
        \exp\left(\sum_{d=1}^{d_i}U_{p_{i,d},(i,d)}^{\gamma}(\sigma ; \zeta)+ h_i(\sigma_i)
            \right)
   \right) \\
=   \frac{1}{N}
   \frac{1}{|\mathcal{M}|}\sum_{\mathfrak{m}\in\mathcal{M}}
   \mathbb{E}\log T_0
   \left(
   \prod_{i\in\mathbf{V}}\sum_{\sigma=\pm 1}
        \exp\left(\sum_{d=1}^{d_i}U_{p_{i,d},(i,d)}^{\gamma}(\sigma ; \zeta)+ h_i(\sigma_i)
            \right)
   \right) \\
=   \frac{1}{N}
   \frac{1}{|\mathcal{M}|}\sum_{\mathfrak{m}\in\mathcal{M}}
   \sum_{i\in\mathbf{V}}
   \mathbb{E}\log T_0
   \left(
   \sum_{\sigma=\pm 1}
        \exp\left(\sum_{d=1}^{d_i}U_{p_{i,d},(i,d)}^{\gamma}(\sigma ; \zeta)+ h_i(\sigma_i)
            \right)
   \right) \\
\end{multline}
And with the same calculations as in the proof of Theorem \ref{thmrs},
for $d, (p_k)$ random variables as defined in Theorem \ref{thmstep},
\begin{equation}
\label{eq:ung-U2}
\widetilde{F}_N = 
    \mathbb{E}
        \log T_0 \left(\sum_{\sigma=\pm 1}\exp\left(\sum_{k=1}^dU^{\gamma}_{p_k,k}(\sigma ; \zeta) + h(\sigma)\right)\right) 
        +o_N(1)\\
\end{equation}
Since the right hand sides of Equations (\ref{eq:ung-E}, \ref{eq:ung-U}, \ref{eq:ung-U2}) do not depend on any specific
choice of $\gamma$, we get the statement of Theorem \ref{thmstep} which does not depend on $\Gamma$. Theorem \ref{th:1rsb}
is a direct consequence of Theorem \ref{thmstep}.
\end{proof}

\bibliographystyle{plain}
\bibliography{bound}

\end{document}